\documentclass[a4paper, 
12pt, reqno]{amsart}

\usepackage{latexsym,amssymb,amsmath,epsfig,amsfonts,graphicx,mathrsfs, amsthm}
\usepackage{bm}
\usepackage{pstricks,pst-node,pst-tree}
\usepackage{tikz}
\usepackage{relsize}
\usepackage{mathabx}
\usepackage{xcolor}
\usetikzlibrary{arrows}
\usepackage{centernot}
\usepackage{enumerate}
\usepackage{booktabs}
\usepackage{yfonts}
\usepackage[foot]{amsaddr}

\usepackage{floatrow}
\newfloatcommand{capbtabbox}{table}[][\FBwidth]

\newcommand{\vc}[1]{\boldsymbol{#1}}

\usepackage{dsfont}

\newcommand{\N}{\mathbb N}

\newcommand{\registered}
   {{\scriptsize \ooalign{\hfil\raise0.07ex\hbox{\scriptsize \sc r}\hfil%
              \crcr\mathhexbox20D}}}

\newtheorem{theorem}{Theorem}

\newtheorem{lemma}{Lemma}
\newtheorem{definition}{Definition}

\newtheorem{corollary}{Corollary}

\newtheorem{remark}{Remark}
\newtheorem{proposition}{Proposition}

\newcommand{\nc}{\newcommand}
\nc{\ds}{\displaystyle}
\nc{\mbZ}{\mathbb Z}
\nc{\mbQ}{\mathbb Q}
\nc{\mbR}{\mathbb R}
\nc{\mbC}{\mathbb C}
\nc{\mbN}{\mathbb N}
\nc{\mbE}{\mathbb E}
\nc{\mbP}{\mathbb P}
\nc{\pr}{\mathbb P}

\nc{\PH}{\emph{PH} }
\nc{\ME}{\emph{ME} }
\nc{\LST}{\emph{LST} }
\nc{\rank}{\mbox{rank\hspace{1pt}}}

\nc{\soph}[1]{\textcolor{violet}{#1}}

\numberwithin{equation}{section}

\numberwithin{figure}{section}

\usepackage[english]{babel}

\begin{document}
\title[Extinction probabilities of Branching Processes]{Extinction probabilities in branching  processes with  countably many types: a general framework}

\author[D.~Bertacchi]{Daniela Bertacchi$^1$}
\address{$^1$ Dipartimento di Matematica e Applicazioni,
Universit\`a di Milano--Bicocca,
via Cozzi 53, 20125 Milano, Italy.}
\email{daniela.bertacchi\@@unimib.it}

\author[P.~Braunsteins]{Peter Braunsteins$^{2,3}$}
\address{$^2$ 
Korteweg-de Vries Instituut,
University of Amsterdam,
Amsterdam, The Netherlands.}
\email{p.t.m.braunsteins\@@uva.nl}
\address{$^3$ School of Mathematics and Statistics,
University of Melbourne, Melbourne, Australia.}
\email{braunsteins.p\@@unimelb.edu.au}

\author[S.~Hautphenne]{Sophie Hautphenne$^3$}
\email{sophiemh\@@unimelb.edu.au}

\author[F.~Zucca]{Fabio Zucca$^4$}
\address{$^4$ Dipartimento di Matematica,
Politecnico di Milano,
Piazza Leonardo da Vinci 32, 20133 Milano, Italy.}
\email{fabio.zucca\@@polimi.it}

\begin{abstract}
We consider Galton--Watson branching processes with countable typeset~$\mathcal{X}$. We study the vectors $\vc q(A)=(q_x(A))_{x\in\mathcal{X}}$ recording the conditional probabilities of extinction in subsets of types $A\subseteq \mathcal{X}$, given that the type of the initial individual is $x$. We first investigate the location of the vectors $\vc q(A)$ in
the set of fixed points of the progeny generating vector and prove that $q_x(\{x\})$  is larger than or equal to the $x$th entry of any fixed point, whenever it is different from 1. Next, we present equivalent conditions for $q_x(A)< q_x (B)$ for any initial type $x$ and $A,B\subseteq \mathcal{X}$. Finally, we develop a general framework to characterise all \emph{distinct} extinction probability vectors, and thereby to determine whether there are finitely many, countably many, or uncountably many distinct vectors. We illustrate our results with examples, and conclude with open questions.
\smallskip


\end{abstract}

\date{\today}
\maketitle

\section{Introduction}
\label{sec:intro}

Branching processes are models for populations where independent individuals reproduce and die. 
If all individuals have the same reproduction law and live in a single location, then the population can be modelled with a \emph{single-type} branching process.
If individuals have specific characteristics (i.e. their location, or in general their ``type") which impact the evolution of the population, then \emph{multitype} branching processes are suitable models.
Here we focus on (discrete-time) \textit{multitype Galton--Watson branching processes} (MGWBPs) with countably many types (where  \textit{countable} includes the finite case as well). These processes arise naturally as stochastic models for various
biological populations (see for instance \cite[Chapter 7]{ka2002}). They can alternatively be interpreted as \emph{branching random walks} (BRWs) on an infinite graph where the types correspond to the vertices of the graph (see for instance \cite{shi} and references therein).

One of the primary quantities of interest in a branching process is the probability that
the population eventually becomes empty or \emph{extinct.} 
Extinction in MGWBPs can be of the whole population (\textit{global extinction}), in all finite subsets of types (\textit{partial extinction}), or more generally, in any fixed subset of types $A$ (\textit{local extinction in $A$}).
To be precise, let $\mathcal{X}$ denote the (countable) typeset, and let $\bm{Z}_n=(Z_{n,x})_{x \in \mathcal{X}}$, where $Z_{n,x}$ records the number of type-$x$ individuals alive in generation $n\geq 0$.
For $A \subseteq \mathcal{X}$, let $\mathcal{E}(A)$ be the event that the process $\{\bm{Z}_n\}_{n\geq 0}$ becomes extinct in $A$, that is, the event that $\lim_{n \to \infty} \sum_{x \in {A}} Z_{n,x} =0$.
Let  $\bm{q}(A)=(q_x(A))_{x \in \mathcal{X}}$ be a vector whose $x$th entry records the conditional probability of local extinction in $A$, given that the population starts with a single type-$x$ individual, that is,
\[
q_x(A):=\mbP(\mathcal{E}(A)\, | \, \bm{Z}_0=\vc e_x),
\]
where $\vc e_x$ is the vector with entry $x$ equal to 1 and all other entries equal to 0.
In particular, note that $\bm{q}(\emptyset)=\bm{1}$.
We let $\bm{q}:= \bm{q}(\mathcal{X})$ be the vector containing the conditional probabilities of global extinction, and we let $\bm{\tilde q} = (\tilde q_x)_{x \in \mathcal{X}}$ be the vector containing the conditional probabilities of partial extinction, where
\[
\tilde q_x := \mbP\bigg( \bigcap_{A: |A| < \infty} \mathcal{E}(A) \, \Big| \, \bm{Z}_0=\vc e_x \bigg).
\]
Several authors have studied properties of $\bm{q}$ and $\bm{\tilde q}$; see for instance \cite{cf:BZ,Bra16,Gan05,Spa89} and most other references herein.

If the process is \emph{irreducible}, meaning that an individual of any given type may have a descendant of any other type, then \begin{itemize}\item when $\mathcal{X}$ is \emph{finite}, $\bm{q}=\bm{q}(A)=\bm{\tilde q}$ for all non-empty $A \subseteq \mathcal{X}$, and 
\item when $\mathcal{X}$ is \emph{countably infinite} and $A$ is \emph{finite}, $\bm{\tilde q}=\bm{q}(A)$ (see for instance \cite[Corollary~1]{Bra18}).
\end{itemize}
More generally, for any non-empty $A \subseteq \mathcal{X}$, it is known that 
$$\bm{q} \leq \bm{q}(A) \leq \bm{\tilde q} \leq \bm{1};$$ in addition, these inequalities may be strict (see for instance \cite{cf:BZ2020} and \cite{Bra18}). Thus the vectors $\bm{q}(A)$ are of independent interest. 
Other than the recent work of \cite{Hut20} (which focuses on different questions than those considered here) and references in the remainder of this section, little attention has been paid to properties of the vectors $\bm{q}(A)$.

The vectors $\{\bm{q}(A)\}_{A \subseteq \mathcal{X}}$  are all solutions of a common fixed point equation. More precisely, if $\vc G(\vc s) := (G_x(\vc s))_{x\in\mathcal{X}}$ records the probability generating
function associated with the reproduction law of each type (defined in \eqref{defG}), then  $\bm{q}(A)$ belongs to the set 
\begin{equation}\label{set_fp}S:=\{\vc s\in[0,1]^{\mathcal{X}}:\vc s=\vc G(\vc s)\}.\end{equation}In other words, $\textrm{Ext}\subseteq S$, where 
\begin{equation}\label{ext_set}\text{Ext}:= \{ \bm{q}(A) : A \subseteq \mathcal{X}  \}\end{equation}
is the set of extinction probability vectors. In this paper, we focus on the following three main questions: 
\begin{itemize}
\item[\emph{{(i)}}] Where are the elements of $\text{Ext}$ located in $S$? (Section~\ref{sec:secondlarg})
\item[\emph{{(ii)}}] When does $\bm{q}(A)$ differ from $\bm{q}(B)$ for two sets $A,B\subseteq \mathcal{X}$? (Section~\ref{sec:comparison})
\item[\emph{{(iii)}}] How many distinct elements  does  $\text{Ext}$ contain and can these elements be identified?  (Section~\ref{sec:extprobvec})
\end{itemize}
While the answers to these questions are well established in the finite-type setting, much less is known when there are infinitely many types. 
Below we discuss the background behind each question and the contributions we make in this paper.


\paragraph{Question \emph{(i)}} It is well known that in the finite-type irreducible setting, the set of fixed points $S$ contains at most two elements: $S=\text{Ext}=\{\vc q,\vc1\}$; see for instance \cite[Chapter~2]{harris}. 
When there are countably many types, $\vc q$ is the minimal element of $S$ \cite[Theorem 3.1]{Moy62}. 
More recently, in \cite{Bra19} the authors proved that, for a
class of branching processes with countably
infinitely many types called \textit{lower Hessenberg branching processes} (LHBPs), $\bm{\tilde q}$ is either equal to $\vc 1$ or to the maximal element of $S\setminus\{\vc1\}$. 
Theorem \ref{LargeLoc} of the present paper implies that the same result holds for \emph{general} irreducible MGWBPs. 
In particular, if there is \textit{strong local survival,} that is, if $\vc q=\vc{\tilde{q}}<\vc 1$, then Theorem~\ref{LargeLoc} implies that $S=\text{Ext}=\{\vc q,\vc1\}$, as in the finite-type setting.
In addition, this theorem also applies in the reducible setting as we show in general that, for any fixed point $\vc s\in S$, if $s_x<1$ then $s_x\leq q_x(\{x\})$.

\paragraph{Question \emph{(ii)}} 

Recent work addresses related questions:
in \cite{cf:BZ14-SLS} and \cite{cf:BZ2020}, the authors provide equivalent conditions for $\vc q(A)=\vc q$ for every non-empty $A\subseteq  \mathcal{X}$;
in \cite{Bra18}, the authors give sufficient conditions for $\vc q(A)\leq \vc q(B)$  that apply to any MGWBP  and  $A,B\subseteq  \mathcal{X}$, as well as sufficient conditions for $\vc q<\vc q(A)<\bm{\tilde q}$ that apply to \emph{block LHBPs}. 
In Theorem  \ref{th:1} we present a number of necessary and sufficient conditions for $q_x(A)< q_x (B)$ for any initial type $x$; this is a significant improvement on \cite[Theorem 3.3]{cf:BZ14-SLS} and \cite[Theorem 2.4]{cf:BZ2020} (see Section~\ref{sec:comparison} for details). One condition in Theorem \ref{th:1} is the existence of an initial type from which, with positive probability,
the process survives in $A$ without ever visiting $B$; another  is the existence of a sequence of types $\{x_n\}_{n\in\N}$ such that \begin{equation}\label{cond11}(1-q_{x_n}(B))/(1-q_{x_n}(A))\rightarrow0\quad\textrm{as $n\to\infty$.}\end{equation}
A consequence of \eqref{cond11} is that, for any extinction probability vector $\bm q(A)\neq\bm q$, we have $\sup_{x\in\mathcal X}q_x(A)=1$ (Corollary \ref{cor:1}). In particular, if all the entries of $\vc{\tilde{q}}$ are uniformly bounded away from 1, then  there is strong local survival ($\vc q=\vc{\tilde{q}}<\vc 1$; Corollary \ref{cor:suff-quasi}).



\paragraph{Question \emph{(iii)}} 
When $\vc q<\vc{\tilde{q}}$, the set of extinction probability vectors $\text{Ext}$ may contain more than two distinct elements. For instance, in processes that exhibit \textit{non-strong local survival} ($\vc q<\vc{\tilde{q}}<\vc 1=\vc q(\emptyset)$), $\text{Ext}$ contains at least three distinct elements; see for instance \cite{cf:BZ14-SLS,cf:GMPV09,cf:MenshikovVolkov,cf:Muller08-2} for examples of such processes.
In recent years, various examples with more than three extinction probabilities have been constructed: for instance,
\cite{Bra18} contains examples with four and five distinct extinction probability vectors.
The set $\text{Ext}$ can even contain uncountably many distinct elements, as shown in
 \cite[Section~3.1]{cf:BZ2020}.
In the same paper, the authors leave open the question of whether $\text{Ext}$ can be \emph{countably} infinite. Up to this point, the literature has focused primarily on specific examples. Here our goal is to 
develop a unified theory to characterise ---and thereby count--- the distinct elements of $\text{Ext}$.

We start by restricting our attention to a more manageable subset of Ext,
\begin{equation}\label{extA}\text{Ext}(\mathcal{A}):= \{ \bm{q}(A) : A \in \Sigma(\mathcal{A}) \},\end{equation}where $\mathcal{A}=\{A_i\}_{i \in K_\mathcal{A}}$ is a (finite or infinite) collection of subsets of $\mathcal{X}$ and $\Sigma(\mathcal{A})$ is the smallest $\sigma$-algebra on $\bigcup_{A \in \mathcal{A}} A$ containing $\mathcal{A}$.
The idea is to select $\mathcal{A}$ carefully so that \emph{(1)} either $\text{Ext} \equiv \text{Ext}(\mathcal{A})$
or $\text{Ext}(\mathcal{A})$ highlights some property of Ext, and \emph{(2)} $\mathcal{A}$ satisfies some minor assumptions, in which case  we call  $\mathcal{A}$ \textit{regular}. We show that we can associate a directed graph $G_\mathcal{A}$ to $\mathcal{A}$, and that if $\mathcal{A}$  is regular, then the analysis of $\text{Ext}(\mathcal{A})$ reduces to the analysis of $G_\mathcal{A}$. More specifically, in this graph, the vertices are the elements of $K_\mathcal{A}$, and there is a directed edge from  $i$ to  $j$ if and only if $\bm{q}(A_i)\geq \vc q(A_j)$, where these pairwise relationships can be determined using  Theorem  \ref{th:1}. We show that there is an injective function from the set of edgeless subgraphs of $G_\mathcal{A}$ to the distinct elements of $\text{Ext}(\mathcal{A})$ (Theorem \ref{th:equivalence} and Lemma~\ref{lem:primitive}); furthermore, if $G_\mathcal{A}$ does not contain a path of infinite length (i.e. an \emph{ascending chain} as defined on Page \pageref{CA}),
we prove that this function is bijective (Theorem \ref{th:equivalence} and Proposition~\ref{pro:cardinality}\emph{(ii)}). 
If $G_\mathcal{A}$ contains ascending chains, then we show that the
 set of edgeless subgraphs can be extended so as to define a bijection between this extended set and the distinct elements of $\text{Ext}(\mathcal{A})$ (Theorem \ref{th:equivalence} and Proposition~\ref{pro:cardinality}\emph{(i)}).
These results translate problems about the distinct extinction probability vectors into much simpler problems about the graph $G_\mathcal{A}$. We use this framework  to provide necessary and sufficient conditions for $\text{Ext}(\mathcal{A})$ to contain finitely many, countably many, or uncountably many distinct elements (Theorem~\ref{th:finiteinfinite2}). To provide a rigorous exposition, we introduce an equivalence relation $\sim$ on the set $2^{K_\mathcal{A}}$ (see Definition~\ref{def:equivalence}) and  then study properties of the quotient set $2^{K_\mathcal{A}}/_\sim$. 


We apply our results to three examples. In Example 1, we consider a specific family of irreducible branching processes where, by varying a single parameter, 
we can transition smoothly between cases where the process has \emph{any finite} number of extinction probability vectors, a \emph{countably infinite} number of extinction probability vectors, and an \emph{uncountable}  number  of extinction probability vectors
(Proposition \ref{pro:ultimate}). This resolves the open question in  \cite{cf:BZ2020}. In Examples 2 and 3, we use our general framework to list all distinct extinction probability vectors in two non-trivial examples: in Example 2, the number of distinct elements of $\text{Ext}(\mathcal{A})$ is the same as the  number of edgeless subgraphs in $G_\mathcal{A}$, while in Example 3, the number of distinct elements of $\text{Ext}(\mathcal{A})$ is strictly larger than the  number of edgeless subgraphs in $G_\mathcal{A}$.



%


The paper is structured as follows. In Section~\ref{sec:notation} we introduce some definitions and notation, as well as some preliminary results.
In Sections~\ref{sec:secondlarg} and \ref{sec:comparison} we tackle Questions~\emph{(i)} and \emph{(ii)}, respectively.
In Section~\ref{sec:extprobvec} we deal with Question \emph{(iii)}; more precisely, in Section~\ref{subsec:indexedsets} we introduce the concept of a regular family $\mathcal{A}$, in Section~\ref{subsec:quotient}
we define the equivalence relation $\sim$ on  $2^{K_\mathcal{A}}$ and establish 
the relationship between $2^{K_\mathcal{A}}/_\sim$ and the distinct elements in $\text{Ext}(\mathcal{A})$,
in Sections~\ref{subsec:primitive} we investigate the structure of the equivalence classes, 
and in Section \ref{subsec:cardinality} we provide conditions for the number of distinct elements in $\text{Ext}(\mathcal{A})$ to be finite, countably infinite, or uncountable. 
In Section~\ref{sec:Ex} we present our examples, and in Section~\ref{sec:open} we discuss open questions. All the proofs, along with some technical lemmas, can be found in Section~\ref{sec:proofs}. In a final appendix, we propose an iterative method to compute the extinction probability vector $\vc q(A)$ for any  $A\subseteq \mathcal{X}$.

In this paper, we let $\vc 1$ denote the infinite column vectors of $1$s.  For any vectors $\vc x$ and $\vc y$, we write $\vc x\leq \vc y$ if $x_i\leq y_i$ for all $i$, and $\vc x< \vc y$ if $\vc x\leq  \vc y$ with $x_i<y_i$ for at least one entry $i$. Finally, we use the shorthand notation $\mbP_x(\cdot):= \mbP(\cdot| \bm{Z}_0=\vc e_x)$ and $\mbE_x(\cdot):= \mbE(\cdot| \bm{Z}_0=\vc e_x)$. We remark that, unless otherwise explicitly stated, our results hold for any generic, not necessarily irreducible, MGWBP.

\section{Preliminaries}\label{sec:notation}



\subsection{Definitions}

In an MGWBP $\{\vc{Z}_n\}_{n\geq 0}$ with countable typeset $\mathcal{X}$, each individual lives for one generation and, at death, independently gives birth to a (finite) random number of offspring. For $x\in\mathcal{X}$ and $\vc{j} = (j_y)_{y\in\mathcal{X}} \in \mathbb{N}^\mathcal{X}$,
let $p_{x\vc{j}}$ denote the
probability that an individual of type~$x$ gives birth to~$j_y$
children of type $y$, for all $y\in\mathcal{X}$. The associated probability
 generating function is 
\begin{equation}\label{defG}
G_x(\vc{s}) := \sum_{\vc j:|\vc j |<\infty}p_{x\vc{j}} \vc{s}^{\vc{j}} := \sum_{\vc j:|\vc j |<\infty}p_{x\vc{j}} \prod_{y\in\mathcal{X}} s_y^{j_y},\qquad\vc{s} \in [0,1]^{\mathcal{X}},
\end{equation}  
where $|\vc j|:=\sum_{y \in \mathcal{X}} j_y$,
and  we  let $\vc G(\vc{s}) := (G_x(\vc s))_{x\in\mathcal{X}}$. Note that $\bm G\colon [0,1]^{\mathcal X}\to  [0,1]^{\mathcal X}$ is nondecreasing and continuous with respect
to the \textit{pointwise convergence} (or \textit{product}) \textit{topology} on $[0,1]^\mathcal{X}$. 
Let $m_{xy}:=\mbE_x[{Z}_{1,y}]=(\partial G_x(\vc{s})/\partial s_y ) |_{\vc{s} = \vc{1}} $ be the expected number of offspring of type $y$ born to a parent of type~$x$, and
let $(\mathcal{X},E_\mathcal{X})$ be the directed graph with vertex set $\mathcal{X}$ and edge set $E_\mathcal{X}=\{(x,y)\in\mathcal{X}^2:m_{xy}>0\}$.
We write $x \to y$ if there is a path from $x$ to $y$ in $(\mathcal{X},E_\mathcal{X})$, and we write $x \leftrightarrow y$ if $x \to y$ and $y \to x$. Note that $x \leftrightarrow x$ because there is always a path of length zero from $x$ to itself.
The equivalence 
class $[x]_{\leftrightarrow}$ of $x$ with respect to $\leftrightarrow$ is called the \textit{irreducible class of $x$}.

The MGWBP $\{\vc{Z}_n\}$ is called  \textit{irreducible} if and only if
the graph $(\mathcal{X},E_\mathcal{X})$ is \textit{connected} (that is, there is only one irreducible class), 
otherwise it is \textit{reducible}.
We say that the process is \textit{non-singular} if, in every irreducible class, there is at least one type whose probability of having exactly one child in that irreducible class is not equal to 1, or, in other words, if
for every $x$, there exists $y \leftrightarrow x$ such that $\pr_y(\sum_{w \leftrightarrow y} {Z}_{1,w}=1)<1$. This assumption is different from the usual one (which is, for every $x$ there exists $y \leftrightarrow x$ such that $\pr_y(\sum_{w \in \mathcal{X}} {Z}_{1,w}=1)<1$), but both definitions are equivalent for irreducible processes.

\subsection{Properties of $\bm{q}(A)$}

For $n\geq 0$ and $A \subseteq \mathcal{X}$, we define $\bm{q}^{(n)}(A):=(q^{(n)}_x(A))_{x \in \mathcal{X}}$
where $$q^{(n)}_x(A)=\mbP_x\left(\sum_{\ell\geq n} \sum_{y \in {A}} Z_{\ell,y} =0\right)$$ is
the probability of extinction in $A$ before generation $n$, starting with a single type-$x$ individual. The sequence
$\{\bm q^{(n)}(A)\}_{n \geq0}$ is (pointwise) nondecreasing,  and satisfies
\begin{equation}\label{eq:extprobab}
 \begin{cases}
 \bm{q}^{(n)}(A)=\vc G(\bm{q}^{(n-1)}(A)),& \quad \forall n \ge 1,\\
 {q}^{(0)}_x(A)=0, &\quad \forall x \in A, \\
 {q}^{(0)}_x(A)=G_x(\bm{q}^{(0)}(A)) &\quad  \forall x \not \in A.\\
\end{cases}
\end{equation}
In addition, $\bm{q}^{(n)}(A)$ converges to $\bm q(A)$ as $n\to\infty$. 
This implies that, for every $A \subseteq \mathcal{X}$, $\bm{q}(A)$ belongs to the set of fixed points $S$ defined in \eqref{set_fp}
(note that $\bm{\tilde q}$ also belongs to $S$).
We observe that  ${q}^{(0)}_x(A)=\mathbb P_x(\mathcal N(A))$, where $\mathcal{N}(A)$ is the event that the process never visits $A$.
In principle, if we knew $\bm{q}^{(0)}(A)$, we could iteratively apply $\bm{G}(\cdot)$ and recover $\bm{q}(A)$ as the limit of the sequence $\bm{q}^{(n)}(A)$.
However, $\bm{q}^{(0)}(A)$ is not uniquely characterised by Equation~\eqref{eq:extprobab}. In other words, $\bm{q}^{(0)}(A)$ is not necessarily the only element of the set of fixed points
 \[
\widehat S^{(A)} := \{ \bm{s} \in [0,1]^\mathcal{X} \colon \bm{s} = \bm{\widehat G}^{(A)}(\bm{s}) \},
\]where the function $\bm{\widehat{G}}^{(A)}\colon [0,1]^{\mathcal X}\to  [0,1]^{\mathcal X}$ is defined by
\[
\widehat {G}_x^{(A)}(\bm{s}) := \begin{cases} 
0, \quad &\text{ if } x \in A \\
G_x(\bm{s}) &\text{ if } x \notin A,
\end{cases}
\]and can be interpreted as the generating function of the offspring distribution in the modified process $\{\bm{\hat{Z}}_n^{(A)}\}$ where types in $A$ produce an infinite offspring number with probability one. Note that, if $A \neq \emptyset$, then $\vc 1\notin\widehat S^{(A)}$. 
For $B \subseteq \mathcal{X}$, we define the probability that the process becomes extinct in $B$ and never visits $A$ as $\bm{q}(B,A):=(q_x(B,A))_{x\in\mathcal X}$, where $q_x(B,A):=\pr_x(\mathcal E(B)\cap \mathcal N(A))$. The vectors $\bm{q}(B,A)$ belong to
$\widehat {S}^{(A)}$ for all $B$ (by the same arguments as those used to show $\bm{q}(A)\in S$). The following result characterizes $\bm{q}^{(0)}(A)$.

\begin{proposition}\label{pro:maximal}
The vectors $\bm{q}(\mathcal{X},A)$ and $\bm{q}^{(0)}(A)\equiv\bm{q}(\emptyset,A)$ are the (componentwise) minimal and {maximal} element of $\widehat S^{(A)}$ respectively. 
\end{proposition}

Observe that  $\bm q^{(0)}(A)$ is uniquely identified by Equation~\eqref{eq:extprobab}  if and only if $\widehat S^{(A)} $ is a singleton, which, by Proposition~\ref{pro:maximal}, occurs if and only if $\bm{q}(\mathcal{X},A)=\bm{q}^{(0)}(A)$;  conditions for $\widehat S^{(A)} $ to be a singleton are given in Theorem \ref{th:1}.
%
 We also point out that, in the irreducible case, $\bm{q}^{(0)}(A)$ can be interpreted as the partial extinction probability vector of $\{\bm{\hat{Z}}_n^{(A)}\}$; in practice, $\bm{q}^{(0)}(A)$ can then be computed numerically using the method developed in \cite{haut2013}, and $\bm{q}(A)$ can be approximated by functional iteration, however it is unclear whether this algorithm converges. An alternative iterative method to compute the vector $\bm{q}(A)$ for any $A \subseteq \mathcal{X}$ can be found in the Appendix.




\section{The second largest fixed point}
\label{sec:secondlarg}

It is well known that $\bm{q}$ is the  componentwise minimal element of $S$ while, 
clearly, $\bm{1}$ is the maximal. The next theorem gives an upper bound, namely $q_x(\{x\})$, for the $x$th component of any fixed point, whenever it is different from $1$.
In the irreducible case, we then have that  $\bm{\tilde q}$ is either the largest or second largest element of $S$: the largest when $\bm{\tilde q}=\bm{1}$, and the second largest when $\bm{\tilde q} < \bm{1}$
(indeed, by \cite[Corollary 4.1]{Bra18}, ${\tilde q}_x={q}_x(\{x\})$). 

\begin{theorem}\label{LargeLoc}
Suppose $\{ \bm{Z}_n \}_{n\geq0}$ is a non-singular \mbox{MGWBP}. 
If $\bm{s} \le \bm{G}(\bm{s})$, then
\begin{enumerate}
\item[(i)] for all $x \in \mathcal{X}$, either
$s_x=1$ or $s_x \leq q_x(\{x\})$;
\item[(ii)] if $s_x<1$, then $s_y \le q_y(\{y\})$ for all $y \in \mathcal{X}$ such that $y \to x$;
\item[(iii)] if the process is irreducible and
$\bm{s} \neq \bm{1}$, then
$\bm{s} \le \bm{\tilde q}$.
\end{enumerate}
\end{theorem}

%
%
%
%
%
%

The following corollary gives further insights into the set of fixed points $S$ when $q_x(\{x\})= q_x$ for all $x$; note that $q_x(\{x\})= q_x<1$ is usually called 
\textsl{strong local survival} in $x$.

\begin{corollary}\label{cor:0}Suppose $\{ \bm{Z}_n \}_{n \geq 0}$ is non-singular and 
let $\bm{s} \in S$. 
\begin{enumerate}
 \item If $q_x(\{x\})= q_x$ for all $x$ 
 then, for every $x \in \mathcal{X}$, either $s_x=1$ or $s_x=q_x(\{x\})$. In this case, any fixed point is an extinction probability vector, that is, $\text{Ext}=S$.
 \item If  $\{ \bm{Z}_n \}_{n \geq 0}$ is irreducible and $\bm{s} \neq \bm{1}$,
then $\bm{s}\leq\bm{\tilde q}$. In particular, if $\bm{\tilde q} = \bm{q}$, then $S = \{ \bm{q}, \bm{1} \}$.
\end{enumerate}
\end{corollary}

\section{When is $\bm{q}(A) \neq\bm{q}(B)$?}\label{sec:comparison}

In order for two extinction probability vectors $\bm{q}(A)$ and $\bm{q}(B)$ to be different, it is necessary for the process to have a positive chance of survival in the symmetric difference of the sets $A$ and $B$. More formally, letting $\mathcal{S}(A):=\mathcal{E}(A)^c$ denote the event that the process survives in $A$,
if $\mbP_x(\mathcal{E}(A \bigtriangleup B))=1 $ then $\mbP_x(\mathcal{S}(A))=\mbP_x(\mathcal{S}(A\cap B))=\mbP_x(\mathcal{S}( B))$, that is,
 $$\bm{q}(A)\neq\bm{q}(B)\qquad \Rightarrow\qquad\exists \,x \in \mathcal{X}\quad s.t.\quad\mbP_x(\mathcal{E}(A \bigtriangleup B))<1.$$
A more powerful characterization of $\bm{q}(A) \neq \bm{q}(B)$ is given in the following theorem, which is  a significant improvement over \cite[Theorem 3.3]{cf:BZ14-SLS}, where the equivalence between \textit{(i)} and \textit{(v)} was proved with $A=\mathcal{X}$.

\begin{theorem}\label{th:1}
For any MGWBP 
and $A,B \subseteq \mathcal{X}$, the following statements are equivalent:
\begin{enumerate}[(i)]
\item there exists $x \in \mathcal{X}$ such that $q_{x}(A) < q_{x}(B)$
\item there exists $x \in \mathcal{X}$ such that $q_{x}(A\setminus B) < q_{x}(B)$
\item there exists $x \in \mathcal{X}$ such that $q_{x}(A) < q^{(0)}_{x}(B)$
\item there exists $x \in \mathcal{X}$ such that,
starting from $x$
there is a positive chance of survival in $A$ without ever visiting $B$
\item there exists $x \in \mathcal{X}$ such that, starting from $x$ there is a positive chance of survival in $A$ and extinction in $B$
\item 
\[
\inf_{x \in \mathcal{X}\colon q_x(A)<1} \frac{1-q_x(B)}{1-q_x(A)}=0.
\]
\end{enumerate}
Moreover, if $A=\mathcal X$ then each of the above conditions is equivalent to 
\begin{itemize}
\item[\emph{(vii)}] $\hat S^{(B)}$ is not a singleton.
\end{itemize}
\end{theorem}
Note that the equivalence between \textit{(i)} and \textit{(iii)} was proved  in \cite[Theorem~2.4]{cf:BZ2020}.



\begin{corollary}\label{cor:1}
 For any  MGWBP, 
 every extinction probability vector
$\bm{q}(A) \neq \bm{q}$, satisfies
$\sup_{x \in \mathcal{X}} q_x(A) =1$.
\end{corollary}

\begin{remark}
In \cite[Lemma 3.3]{Moy62}, the author showed that, if an MGWBP is irreducible, all fixed points $\bm{s}\neq \bm{q}$ with $\inf_{x \in \mathcal{X}} s_x >0$ satisfy $\sup_{x \in \mathcal{X}} s_x =1$. However, the condition \mbox{`$\inf_{x \in \mathcal{X}} s_x >0$'} was described as unsatisfactory.
Corollary \ref{cor:1} proves that all extinction probabilities $\bm{q}(A) \neq \bm{q}$ satisfy $\sup_{x \in \mathcal{X}} q_x(A) =1$ under no assumptions (not even irreducibility).  
\end{remark}

\noindent
In the irreducible case, Corollary~\ref{cor:1} easily implies the following result.

\begin{corollary}\label{cor:suff-quasi}
Suppose that $\{ \bm{Z}_n \}$ is irreducible. 
If $\sup_{x \in \mathcal{X}} \tilde q_x <1$ then $\bm{\tilde q} = \bm{q}$ and $S= \{ \bm{q}, \bm{1} \}$.
\end{corollary}


Corollary~\ref{cor:suff-quasi} applies to irreducible quasi-transitive MGWBPs (see for instance \cite[Section 2.4]{cf:BZ14-SLS} for the definition) where
$\bm{\tilde q} < \bm{1}$, extending
\cite[Corollary 3.2]{cf:BZ14-SLS}; indeed,
in that case the coordinates of $\bm{\tilde q}$ take their value in a finite
set and they are all different from 1. It also applies to MGWBPs with an absorbing barrier (see \cite{Big91}) with $\bm{\tilde q}<\bm{1}$, for which $\mathcal{X}=\mbN$ and $\tilde q_x$ is decreasing in $x$.

\section{The set of extinction probability vectors}\label{sec:extprobvec}

We now turn our attention to the set Ext of extinction probability vectors. Our analysis builds upon an important consequence of Theorem \ref{th:1} (which we state in Corollary \ref{cor:equivalence}).
We start by defining relations between subsets $A,B \subseteq \mathcal{X}$ in a given 
MGWBP: we write 
\begin{itemize}
\item $A \Rightarrow B$ if survival in $A$ implies survival in $B$ 
from every starting point (that is, $\mathbb P_x (\mathcal{S}(B)\,|\,\mathcal S(A))=1$ for all $x \in \mathcal{X}$),
\item $A \nRightarrow B$ if there is a positive chance of survival in $A$ and extinction in $B$ from some starting points (that is, $\mathbb P_x (\mathcal{S}(B)\,|\,\mathcal S(A))<1$)
for some $x \in \mathcal{X}$),
\item $A \Leftrightarrow B$ if survival in $A$ implies survival in $B$ and vice-versa
from every starting point,
\item $A \nLeftrightarrow B$ if there is a positive chance of survival in $B$ and extinction in $A$ from some starting points and vice-versa.
\end{itemize}
Note that $A \Leftrightarrow A$ for all $A \subseteq \mathcal{X}$. 
The next corollary is a straightforward consequence of the equivalence between \textit{(i)} and \textit{(v)} in Theorem \ref{th:1}.
\begin{corollary}\label{cor:equivalence}
Let $A,B \subseteq \mathcal{X}$.
\begin{enumerate}
 \item $A \Rightarrow B$ if and
 only if $\bm{q}(A) \ge \bm{q}(B)$.
 \item  $A \Leftrightarrow B$ if and
 only if $\bm{q}(A)  = \bm{q}(B)$.
 \item $A \nLeftrightarrow B$ if and only if there is no order relation between $\bm{q}(A)$ and $\bm{q}(B)$.
\end{enumerate}
\end{corollary}
We point out that any of the six equivalent conditions in Theorem \ref{th:1} can be used to establish the relation between the pair $A,B \subseteq \mathcal{X}$.

\subsection{Regular families of subsets}
\label{subsec:indexedsets}

We will use the pairwise relations between subsets of $\mathcal{X}$ to study Ext. Rather than considering all subsets of $\mathcal{X}$, it is often sufficient to restrict our attention to a particular family of subsets. 
More precisely,
we focus on
\[
\text{Ext}(\mathcal{A}):= \{ \bm{q}(A) \colon A \in \Sigma(\mathcal{A}) \},
\] 
where $\mathcal{A}=\{ A_1,A_2, \dots, A_{\kappa_{\mathcal{A}}} \}$, with $\kappa_{\mathcal{A}} \leq \infty$, $A_i \subseteq \mathcal{X}$ for all $i \in K_{\mathcal{A}}:= \{1, \dots, \kappa_{\mathcal{A}} \}$, and $\Sigma(\mathcal{A})$ is the smallest $\sigma$-algebra on $\bigcup_{i \in K_\mathcal{A}} A_i$ containing all $A_i$. 
The idea is to select a suitable family $\mathcal{A}$ so that either $\text{Ext} \equiv \text{Ext}(\mathcal{A})$ as in the examples in Section \ref{sec:Ex}, and in \cite[Section 3.2]{cf:BZ2020} and \cite[Example 1]{Bra18}, or so that $\text{Ext}(\mathcal{A})$ highlights some property of Ext as in \cite[Section 3.1]{cf:BZ2020}. Below we show that the analysis of $\text{Ext}(\mathcal{A})$ is substantially simpler under some minor regularity conditions on $\mathcal{A}$ and the associated MGWBP.

%
%
%
%
%
%

\begin{definition}\label{def:regular} We call $\mathcal{A}$ \emph{regular} if
\begin{itemize}
\item[(C1)] for any $i \neq j\in K_\mathcal{A}$, we have $A_i \cap A_j = \emptyset$;
\item[(C2)]for any $i \in K_\mathcal{A}$, we have $\bm{q}(A_i)<\bm{1}$;
\item[(C3)] there does not exist $i \neq j\in K_{\mathcal{A}}$ such that $A_i \Leftrightarrow A_j$;
\item[(C4)] if $ A \in \Sigma(\mathcal{A})$ and 
${I}_A:= \{ i \in K_{\mathcal{A}} \colon A_i \Rightarrow A \} \neq \emptyset$ then $\bigcup_{i \in {I}_A} A_i \Rightarrow A$; 
\item[(C5)] if $ i \in K_{\mathcal{A}}$  and ${J}_i := \{ j \in K_{\mathcal A}\colon A_i\nRightarrow A_j\} \neq \emptyset$ then $A_i \nRightarrow \bigcup_{j \in {J}_i} A_j$.
\end{itemize}
\end{definition}

Condition (C1) allows an easy description of $\Sigma(\mathcal{A})$ in terms of unions of sets in $\mathcal{A}$; in particular, under this condition, $I \mapsto  \bigcup_{i \in I} A_i$ is a surjective map from $2^{K_{\mathcal{A}}}$ onto $\Sigma(\mathcal{A})$. If in addition (C2) holds, then $A_i \neq \emptyset$ for all $i \in K_\mathcal{A}$ and the map is also injective.
Conditions (C2) and (C3) can be viewed as a preprocessing step which removes elements from $\mathcal{A}$ that lead to non-distinct extinction probability vectors. In particular we observe that (C3) ``almost implies'' (C2), meaning that, if (C3) holds then $\bm{q}(A_i)=\bm{1}$ for at most one $i \in K_\mathcal{A}$ (by Corollary \ref{cor:equivalence}). Thus, (C2) implies that $\bm{q}\big ( \bigcup_{i \in I} A_i \big )=\bm{1}$ if and only if $I=\emptyset$, in particular $\emptyset \not \in \mathcal{A}$.
Conditions (C4) and (C5) are minor regularity assumptions that 
we use to compare  the number of distinct elements in $\text{Ext}(\mathcal{A})$ and the cardinality of the quotient set of $2^{K_\mathcal{A}}$ with respect to a suitable equivalence relation (see Definition~\ref{def:equivalence}).
On the other hand, (C2) and (C3) allow us to study the cardinality of a particular subset of this quotient set (see Definition~\ref{def:maximal} and Equation~\eqref{eq:primitive}).

\subsection{Equivalent subsets of indices}
\label{subsec:quotient}

Not all the elements of $\text{Ext}(\mathcal{A})$ are necessarily distinct. For instance,
if
$A_i \Rightarrow A_j$, then  $\bm{q} (A_i\cup A_j)=\bm{q} (A_j)$. This motivates the next definition.


\begin{definition}\label{def:equivalence}
The subsets $I,J \subseteq K_{\mathcal{A}}$ are \emph{equivalent}, and we write
$I \sim J$, if and only if 
\begin{enumerate}[(i)]
 \item for every $i \in I$ there exists $j \in J$ such that $A_i \Rightarrow A_j$, and
  \item for every $j \in J$ there exists $i \in I$ such that $A_j \Rightarrow A_i$. 
\end{enumerate}
\end{definition}
\noindent
Observe that $\emptyset \sim I$ implies $I=\emptyset$.

We are interested in the number of \emph{distinct} elements in $\text{Ext}(\mathcal{A})$, which we denote by $|\text{Ext}(\mathcal{A})|$. 
The next theorem implies that, if $\mathcal{A}$ is regular, then $|\text{Ext}(\mathcal{A})|$ equals the cardinality of the quotient set $2^{K_{\mathcal{A}}}/_\sim $, that is, the number of equivalence classes.

\begin{theorem}\label{th:equivalence}
Given a family $\mathcal{A}$ and $I,J \subseteq K_{\mathcal{A}}$, consider the following relations:
 \begin{enumerate}
  \item[(i)] $I \sim J$
\item[(ii)] $\bigcup_{i \in I} A_i \Leftrightarrow 
\bigcup_{j \in J} A_j$
\item[(iii)] $\bm{q} \Big ( \bigcup_{i \in I} A_i\Big)=\bm{q} \Big ( \bigcup_{j \in J} A_j\Big)$.
 \end{enumerate}
Then $(ii) \Leftrightarrow (iii)$. \\
If 
(C4) holds then $(i) \Rightarrow (iii)$; if in addition (C1) holds then $|\text{Ext}(\mathcal{A})| \leq  |2^{K_{\mathcal{A}}}/_\sim|$.\\ 
If (C2) and (C5) hold then $(iii) \Rightarrow (i)$ and $|\text{Ext}(\mathcal{A})| \geq  |2^{K_{\mathcal{A}}}/_\sim|$.
\end{theorem}

%


\subsection{Primitive subsets and ascending chains}
\label{subsec:primitive}

In order to characterize $|\text{Ext}(\mathcal{A})|$, the next step is to better understand the structure of the equivalence classes. To help visualise these classes,
we associate a directed graph $G_{\mathcal{A}} = (K_{\mathcal{A}}, E_{\mathcal{A}})$, with 
edge set $E_{\mathcal{A}}:=\{ (i,j)\in K_{\mathcal{A}}^2 \colon A_i \Rightarrow A_j \}$,
to a given MGWBP and family $\mathcal{A}=\{ A_1,A_2, \dots, A_{\kappa_{\mathcal{A}}} \}$. 
Observe that
\begin{itemize}
\item[(P1)]  $(i,j),(j,k) \in {E}_{\mathcal{A}}$ implies  $(i,k) \in {E}_{\mathcal{A}}$ (by transitivity of the relation $\Rightarrow $),
\item[(P2)] $(i,i) \in E_\mathcal{A}$ for all $i \in K_\mathcal{A}$,
\end{itemize}
and, under the regularity condition (C3),
\begin{itemize}
\item[(P3)] $G_{\mathcal{A}}$ contains no cycles (of length greater than one).
\end{itemize}
Note that in $G_{\mathcal{A}}$, there is a path from $i$ to $j$ if and only if $(i,j) \in E_\mathcal{A}$. The next lemma states that,
given a directed graph $(X, E_X)$ satisfying (P1) and (P3),
there exist an MGWBP and a regular family $\mathcal{A}$ such that $G_{\mathcal{A}}=(X, E_X).$
\begin{lemma}\label{lem:canonic}
 Let $(Z,E_Z)$ be a directed graph where 
 \begin{itemize}
 \item $Z$ is at most countable, 
 \item 
there are no cycles (closed paths).
\end{itemize}
 Then there exists an MGWBP and a regular family $\mathcal{A}=\{A_i\}_{i \in Z}$ such that $A_i \Rightarrow A_j$ if and only if 
there is a path from $i$ to $j$ in $(Z,E_Z)$.
\end{lemma}

For any subset $I \subseteq K_{\mathcal{A}}$, we define the subgraph induced in $G_{\mathcal{A}}$ by $I$ as $$G_{\mathcal{A}}[I] := (I, E_{\mathcal{A}}[I]),\quad \textrm{with $E_{\mathcal{A}}[I] := \{ (i,j) \in I^2 \colon A_i \Rightarrow A_j \}$}.$$ 

\begin{definition}\label{def:maximal}
We call $I \subseteq K_\mathcal{A}$ \emph{primitive} if
 for all $i,j \in I$, $i \neq j$, we have $A_i \nLeftrightarrow A_j$. Equivalently, a subset $I$ is primitive if the induced subgraph $G_{\mathcal{A}}[I]$ is edgeless. We write $\textswab{P}_\mathcal{A}$ for the set of primitive subsets of $K_\mathcal{A}$.
 
 \end{definition}
 
The following properties are straightforward:
 \begin{itemize}
 \item $I:=\emptyset$ is primitive and, if (C2) holds, it is the only subset of $K_\mathcal{A}$ such that 
 $\bm{q} \big (\bigcup_{i \in I} A_i \big )=\bm{1}$;
  \item every singleton $\{i\}$ is primitive.
  \item every subset of a primitive subset is primitive;
  \item if $\{I_n\}_n$ is a sequence of primitive subsets  of $K_\mathcal{A}$ such that $I_n \subseteq I_{n+1}$ (for all $n$) then $\bigcup_{n} I_n$ is primitive.
 \end{itemize}


From the definition of $\sim$, if (C3) holds, then the equivalence class of a primitive subset $I$ is
 \begin{equation}\label{eq:primitive}
  [I]_\sim=
  \Big\{J \subseteq K_\mathcal{A}\colon J \supseteq I, \forall j \in J,\, \exists
  i \in I, \,A_j \Rightarrow A_i\Big\}.
 \end{equation}
 %
 In particular, given two primitive subsets
 $I_1 \neq I_2$ we have 
 $[I_1]_\sim
 \neq
 [I_2]_\sim$. Hence $\textswab{P}_\mathcal{A}$ can be identified with a (possibly proper) subset of $2^{K_\mathcal{A}}/_\sim$. This directly leads us to the next result about the map
$$f_\mathcal{A}:\textswab{P}_\mathcal{A}\to 2^{K_\mathcal{A}}/_\sim\quad \textrm{s.t.}\quad f_\mathcal{A}(I)=[I]_\sim.$$
 
  \begin{lemma}\label{lem:primitive}
 If $\mathcal{A}$ satisfies (C3) then  $f_\mathcal{A}$
 is injective; in particular 
 $|\textswab{P}_\mathcal{A}| \leq |2^{K_\mathcal{A}}/_\sim|$. 
 \end{lemma}
 
 We will see that in many situations, the injective map $f_\mathcal{A}$ is actually bijective, in which case, if $\mathcal{A}$ is regular, then by Theorem \ref{th:equivalence} there is a one-to-one correspondence between the distinct extinction probability vectors in Ext$(\mathcal{A})$ and the primitive subsets. We now present two illustrative examples: in Figure  \ref{pic1}, $f_\mathcal{A}$ is bijective, and in  Figure  \ref{pic2}, $f_\mathcal{A}$ is not surjective because no primitive subset belongs to the equivalence class of $I=\{3,4,5,\ldots\}$.


\begin{figure}
\begin{floatrow}
\ffigbox{\hspace{-2.5cm}%
  \begin{tikzpicture}[scale=0.80]

\tikzset{vertex/.style = {shape=circle,draw,minimum size=1.3em}}
\tikzset{edge/.style = {->,> = latex'}}

\node at (2,4.2) {Relations: };
\node at (2,3.3) {$A_1 \Rightarrow A_3$, $A_2 \Rightarrow A_1$};
\node at (2,2.4) {$A_2 \Rightarrow A_3$, $A_2 \Rightarrow A_4$};


\node at (2,1.2) {$G_{\mathcal{A}}$:};

\node[vertex, minimum size=.9cm] (1) at  (0,0) {\footnotesize $1$};
\node[vertex, minimum size=.9cm] (2) at  (4,0) {\footnotesize $3$};
\node[vertex, minimum size=.9cm] (3) at  (0,-4) {\footnotesize $2$};
\node[vertex, minimum size=.9cm] (4) at  (4,-4) {\footnotesize $4$};

\draw[edge, line width=1.5pt] (1) to[bend left=0] (2);
\draw[edge, line width=1.5pt] (3) to[bend left=0] (1);
\draw[edge, line width=0.5pt] (3) to[bend left=0] (2);
\draw[edge, line width=1.5pt] (3) to[bend left=0] (4);

\end{tikzpicture}%
}{%
}
\capbtabbox{\hspace{-2cm}\vspace{0.0cm}%
 \begin{tabular}{ p{1.1cm}||p{6.5cm} }
$\textswab{P}_\mathcal{A}$    & Ext$(\mathcal{A})$ \\[+2mm]
 \hline
$\emptyset$  & $\bm{1}$    \\[+3.5mm]
$\{1\}$  &   $\bm{q}(A_1)=\bm{q}(A_1 \cup A_2)$  \\[+3.5mm]
$\{2\}$ & $\bm{q}(A_2)$ \\[+3.5mm]
$\{3\}$    & $\bm{q}(A_3)=\bm{q}(A_1 \cup A_3)=\bm{q}(A_2 \cup A_3)$\\&$\phantom{\bm{q}(A_3)}=\bm{q}(A_1 \cup A_2 \cup A_3)$\\[+3.5mm]
$\{4\}$ &  $\bm{q}(A_4)=\bm{q}(A_2 \cup A_4)$  \\[+3.5mm]
$\{1,4\}$ & $\bm{q}(A_1 \cup A_4)=\bm{q}(A_1 \cup A_2 \cup A_4)$   \\[+3.5mm]
$\{3,4\}$ & $\bm{q}(A_3 \cup A_4)=\bm{q}(A_2 \cup A_3 \cup A_4)$\\&$\phantom{\bm{q}(A_3 \cup A_4)}= \bm{q}(A_1 \cup A_2 \cup A_3 \cup A_4)$ 
\end{tabular}
}{%
}
\end{floatrow}
\end{figure}
\begin{figure}
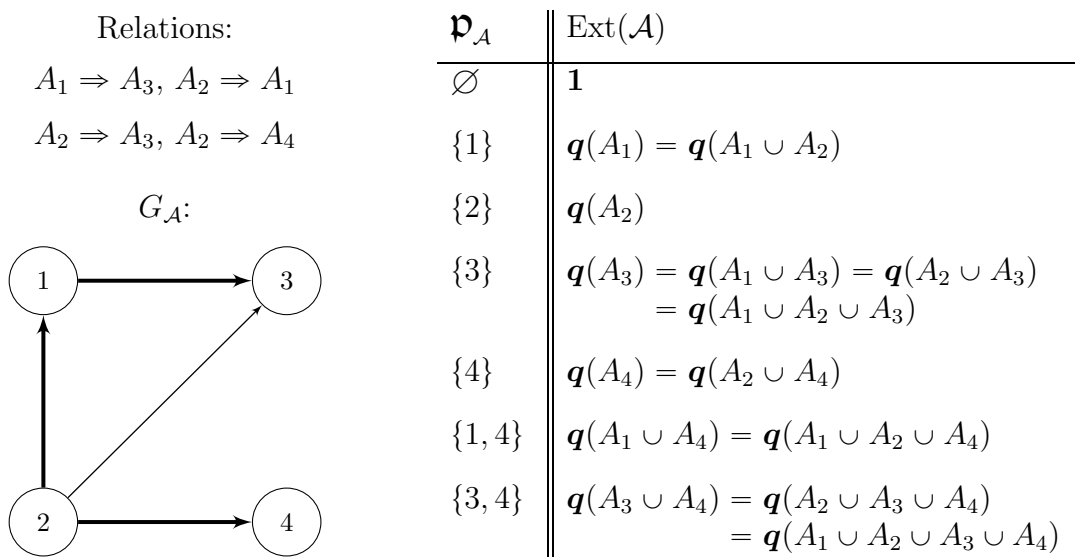
 \caption{\label{pic1}A regular family $\mathcal{A}=\{A_1,A_2,A_3,A_4\}$ with its associated directed graph $G_{\mathcal{A}}$, the set of primitive subsets $\textswab{P}_\mathcal{A}$, and the elements in Ext$(\mathcal{A})$. There is a one-to-one correspondence between the primitive subsets and the distinct elements in Ext$(\mathcal{A})$.}
\end{figure}

 \begin{figure}[h!]
 \begin{tikzpicture}[scale=0.70]

\tikzset{vertex/.style = {shape=circle,draw,minimum size=1.3em}}
\tikzset{edge/.style = {->,> = latex'}}




\node[vertex, minimum size=.9cm] (2) at  (0,0) {\footnotesize $2$};
\node[vertex, minimum size=.9cm] (1) at  (4,0) {\footnotesize $1$};
\node[vertex, minimum size=.9cm] (3) at  (0,-4) {\footnotesize $3$};
\node[vertex, minimum size=.9cm] (4) at  (4,-4) {\footnotesize $4$};
\node[vertex, minimum size=.9cm] (5) at  (8,-4) {\footnotesize $5$};
\node[vertex, minimum size=.9cm] (6) at  (12,-4) {\footnotesize $6$};
\node[vertex, minimum size=.9cm] (7) at  (16,-4) {\footnotesize $7$};

\node (8) at  (18,-4) { };

\node at (18.5,-4) {$\dots$ };

\draw[edge, line width=1.5pt] (2) to[bend left=0] (1);
\draw[edge, line width=0.5pt] (3) to[bend left=0] (1);
\draw[edge, line width=1.5pt] (3) to[bend left=0] (2);

\draw[edge, line width=1.5pt] (3) to[bend left=0] (4);
\draw[edge, line width=1.5pt] (4) to[bend left=0] (5);
\draw[edge, line width=1.5pt] (5) to[bend left=0] (6);

\draw[edge, line width=0.5pt] (3) to[bend left=20] (5);
\draw[edge, line width=0.5pt] (3) to[bend left=30] (6);

\draw[edge, line width=0.5pt] (4) to[bend left=20] (6);

\draw[edge, line width=0.5pt] (3) to[bend left=35] (7);
\draw[edge, line width=0.5pt] (4) to[bend left=30] (7);
\draw[edge, line width=0.5pt] (5) to[bend left=20] (7);
\draw[edge, line width=1.5pt] (6) to[bend left=0] (7);

\draw[edge, line width=1.5pt] (7) to[bend left=0] (8);

\end{tikzpicture}%
\caption{\label{pic2}The directed graph $G_{\mathcal{A}}$ of a regular family $\mathcal{A}=\{A_1,A_2,A_3,\ldots\}$ with an ascending chain. The set of primitive subsets is $\textswab{P}_\mathcal{A}=\{\emptyset,\{i\}_{i=1,2,3,\dots},\{1,j\}_{j=4,5,6,\dots},\{2,j\}_{j=4,5,6,\dots}\},$ and the set of representatives of pure ascending chains is $\textswab{C}_\mathcal{A}=\{\emptyset,\{3,4,5,\dots\}\}$. }
\end{figure}
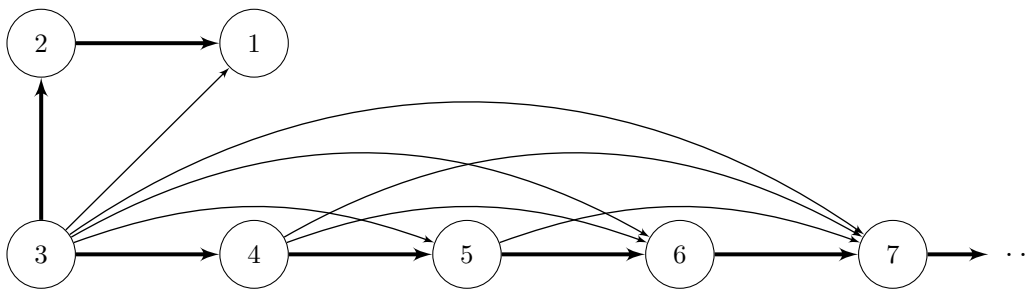

  In order for the map $f_\mathcal{A}$ to be bijective in general, the domain $\textswab{P}_\mathcal{A}$ of $f_\mathcal{A}$ needs to be extended. 
 To understand how to extend $\textswab{P}_\mathcal{A}$, we need a more complete description of the codomain $2^{K_\mathcal{A}}/_\sim$ of  $f_\mathcal{A}$. We consider the following subsets of every $I \subseteq K_{ \mathcal{A}}$:
\[
 \begin{cases}
 I_M:=\{i \in I\colon \forall j \in I, \, j \neq i, \, A_i \nRightarrow A_j \}\\
  I_d:=\{i \in I \colon \exists j \in I_M,\, A_i \Rightarrow A_j\}\\
  I_c:=I \setminus I_d.
 \end{cases}
\]
Roughly speaking,
 $I_M$ contains the vertices with out-degree zero in $G_{\mathcal{A}}[I]$, and  
$I_d$ is the largest subset of $I$ equivalent to $I_M$ (clearly, $I_M \subseteq I_d$, since $A_j \Rightarrow A_j$ for every $j \in K_\mathcal{A}$).  If we think of ``$\Rightarrow$'' as a partial preorder relation ``$\le$'' (it is a partial order relation if (C3) holds), then $I_M$ can be interpreted as the primitive subset of maximal elements of $I$, and $I_d$ as the subset of elements which are smaller than a maximal element.
Finally, $I_c$ is the subset of elements which are not comparable with any maximal element of $I$; in particular,
$$I_c=\{i \in I \colon \not \exists j \in I_M, \, A_i \Rightarrow A_j\}
\equiv \{i \in I \colon \forall j \in I_M, \, A_i \not \Leftrightarrow A_j\}
\equiv \{i \in I \colon \not \exists j \in I_d, \, A_i \Rightarrow A_j\}.$$
As an example,  let $I=K_\mathcal{A}=\mathbb{N}$ for the family $\mathcal{A}$ considered in  Figure \ref{pic2}; then $I_M=\{1\}$, $I_d=\{1,2,3\}$, and $I_c=\{4,5,6,\ldots\}$.

Clearly $I$ is primitive if and only if $I=I_M$; moreover $[I_M]_\sim=[I_d]_\sim$, and if $I \neq \emptyset$ then $[I_d]_\sim \neq [I_c]_\sim$. 
  The next lemma states several other properties of the subsets $I_M, I_d,$ and $I_c$; in particular it extends the representation of the equivalence class of a primitive subset given in~\eqref{eq:primitive} to that of a generic subset (Lemma \ref{lem:IM}\emph{(vii)}).

\begin{lemma}\label{lem:IM}Let $I,J\subseteq K_\mathcal{A}$. \begin{enumerate}[(i)]
\item $ I \sim I_M \Leftrightarrow I_c=\emptyset$;
\item $I_M\sim J_M\Leftrightarrow I_d \sim J_d$;
\end{enumerate}
Suppose that (C3) holds.
\begin{enumerate}[(i)] \setcounter{enumi}{2}
\item  $I_M=J_M \Leftrightarrow I_d \sim J_d$; 
\item $I \sim J \Leftrightarrow \ I_M=J_M $ and $I_c \sim J_c$;
\item $I\sim J$ for some primitive $J$ if and only if $I_c=\emptyset$;
\item if $I_c\neq \emptyset$ then $I_c$ is infinite;
\item 
 $
  [I]_\sim=\{H \cup W\colon H,W \subseteq K_{\mathcal{A}},\, H_d \sim I_d, \,W_c \sim I_c,\, H_c=W_d=\emptyset\}$.
\end{enumerate}
\end{lemma}
%
%
Any infinite sequence $\{i_n\}_{n\geq 0}$ of distinct elements of $K_\mathcal{A}$ such that $A_{i_n} \Rightarrow A_{i_{n+1}}$ will be called an \emph{ascending chain}. Under (C3),  if $I_c$ is non-empty then every element in $I_c$ belongs to an ascending chain (see the proof of Lemma~\ref{lem:IM}\emph{(vi)}).
Any $I\subseteq K_\mathcal{A}$ such that $I=I_c$ (that is, $I_d=I_M=\emptyset$) will be called a \emph{pure} ascending chain.
From Lemma~\ref{lem:IM}\textit{(vii)}, any subset $J$ equivalent to a pure ascending chain is also a pure ascending chain (that is, if $I=I_c$ and $J \sim I$, then $J=J_c$).

Given two equivalent subsets 
$I$ and $J$, observe that
$$\{i\in K_\mathcal{A}: \exists j\in I, A_i\Rightarrow A_j\}= \{i\in K_\mathcal{A}: \exists j\in J,A_i\Rightarrow A_j\}.$$
The largest subset equivalent to $I$, defined as $I^+:=\{i\in K_\mathcal{A}: \exists j\in I, A_i\Rightarrow A_j\}$, is a natural representative of the equivalence class $[I]_\sim$. We let \label{CA}
$$\textswab{C}_\mathcal{A}:=\{I\subseteq K_\mathcal{A}: \exists J=J_c, J^+=I\}$$ be the set of representatives of pure ascending chains; note that $\textswab{C}_\mathcal{A}$ is non-empty since $\emptyset \in \textswab{C}_\mathcal{A}$. Moreover $J \in \textswab{C}_\mathcal{A}$ if and only if $J=J_c$ and $J=J^+$.

Recall that the domain of $f_\mathcal{A}$ is $\textswab{P}_\mathcal{A}$, which is non-empty (since $\emptyset $ is primitive), and that, by Lemma~\ref{lem:primitive}, $f_\mathcal{A}$ is injective (under (C3)).
The following proposition implies that $\textswab{P}_\mathcal{A}$ can be extended by means of $\textswab{C}_\mathcal{A}$ to the set
\begin{equation}\label{eq:Ia}
\textswab{I}_\mathcal{A}:
=
\{(I, J) \in \left(\textswab{P}_\mathcal{A} \times \textswab{C}_\mathcal{A}\right)
\colon  
 I\cap J=\emptyset, \,(J\setminus I^+)^+=J
\}.
\end{equation}
Clearly
$\{\emptyset\} \times \textswab{C}_\mathcal{A}$ and $\textswab{P}_\mathcal{A} \times \{\emptyset\}$ are subsets of $\textswab{I}_\mathcal{A}$; in particular $(\emptyset,\emptyset) \in \textswab{I}_\mathcal{A}$.
We define the map
$$g_\mathcal{A}:2^{K_\mathcal{A}}/_\sim\to\textswab{I}_\mathcal{A}\quad \textrm{s.t.}\quad g_\mathcal{A}([I]_\sim)=(I_M, (I_c)^+).$$



\begin{proposition}\label{pro:cardinality}
 If $\mathcal{A}$ satisfies (C3), then $g_\mathcal{A}$ is bijective; in particular,
 \begin{enumerate}[(i)]
 \item $|2^{K_\mathcal{A}}/_\sim|=|\textswab{I}_\mathcal{A}|$,
 \item if there are no ascending chains (i.e.~$\textswab{C}_\mathcal{A}=\{\emptyset\}$), then the map $f_\mathcal{A}$ is bijective, that is,  every equivalence class contains one (unique) primitive subset. 
\end{enumerate}

\end{proposition}

Note that $f_\mathcal{A}=g^{-1}_\mathcal{A}\circ h$ where $h$ is the natural bijection from $\textswab{P}_\mathcal{A}$ onto  $\textswab{P}_\mathcal{A} \times \{\emptyset\}$. In the example considered in Figure \ref{pic1}, $\textswab{C}_\mathcal{A}=\{\emptyset\}$, hence $\textswab{I}_\mathcal{A}=\textswab{P}_\mathcal{A} \times \{\emptyset\}$, while  in the example considered in Figure \ref{pic2}, $\textswab{C}_\mathcal{A}=\{\emptyset,\{3,4,5,\dots\}\}$, and $$\textswab{J}_\mathcal{A}=\{(I,\emptyset)\colon {I\in\textswab{P}_\mathcal{A}}\} \cup \{(\emptyset,\{3,4,5,\ldots\}),(\{i\},\{3,4,5,\ldots\})_{i=1,2}\}.$$
In Figure \ref{pic3} we provide a modification of the example considered in Figure \ref{pic2} that illustrates why the condition $I\cap J=\emptyset$ is not sufficient in the definition of $\textswab{I}_\mathcal{A}$ in order for $g_\mathcal{A}$ to be bijective: take $I=\{4',5',6',\ldots\}$ and $J=\{3,4,5,\dots\}$; we have $I\cap J=\emptyset$, but $I^+=I\cup J$, so $(J\setminus I^+)^+=\emptyset\neq J$. In this case, $g_\mathcal{A}^{-1}(I,J)=g_\mathcal{A}^{-1}(I,\emptyset)$ because $[I\cup J]_\sim=[I]_\sim$, so $g_\mathcal{A}$ is not bijective.
Additional examples where we identify $\textswab{J}_\mathcal{A}$ are given in Section \ref{sec:Ex}.



 \begin{figure}[h!]
 \begin{tikzpicture}[scale=0.70]

\tikzset{vertex/.style = {shape=circle,draw,minimum size=1.3em}}
\tikzset{edge/.style = {->,> = latex'}}




\node[vertex, minimum size=.9cm] (2) at  (0,-1.5) {\footnotesize $2$};
\node[vertex, minimum size=.9cm] (1) at  (4,-1.5) {\footnotesize $1$};
\node[vertex, minimum size=.9cm] (3) at  (0,-4) {\footnotesize $3$};
\node[vertex, minimum size=.9cm] (4) at  (4,-4) {\footnotesize $4$};
\node[vertex, minimum size=.9cm] (5) at  (8,-4) {\footnotesize $5$};
\node[vertex, minimum size=.9cm] (6) at  (12,-4) {\footnotesize $6$};
\node[vertex, minimum size=.9cm] (7) at  (16,-4) {\footnotesize $7$};

\node (8) at  (18,-4) { };
\node at (18.5,-4) {$\dots$ };

\node[vertex, minimum size=.9cm] (9) at  (4,-6.5) {\footnotesize $4'$};
\node[vertex, minimum size=.9cm] (10) at  (8,-6.5) {\footnotesize $5'$};
\node[vertex, minimum size=.9cm] (11) at  (12,-6.5) {\footnotesize $6'$};
\node[vertex, minimum size=.9cm] (12) at  (16,-6.5) {\footnotesize $7'$};

\node (13) at  (18,-6.5) { };
\node at (18.5,-6.5) {$\dots$ };

\draw[edge, line width=1.5pt] (2) to[bend left=0] (1);

\draw[edge, line width=1.5pt] (3) to[bend left=0] (2);

\draw[edge, line width=1.5pt] (3) to[bend left=0] (4);
\draw[edge, line width=1.5pt] (4) to[bend left=0] (5);
\draw[edge, line width=1.5pt] (5) to[bend left=0] (6);

\draw[edge, line width=1.5pt] (6) to[bend left=0] (7);

\draw[edge, line width=1.5pt] (7) to[bend left=0] (8);

\draw[edge, line width=1.5pt] (4) to[bend left=0] (9);
\draw[edge, line width=1.5pt] (5) to[bend left=0] (10);
\draw[edge, line width=1.5pt] (6) to[bend left=0] (11);
\draw[edge, line width=1.5pt] (7) to[bend left=0] (12);

\end{tikzpicture}%
\caption{\label{pic3}The directed graph $G_{\mathcal{A}}$ of a regular family with an ascending chain (the edges implied by transitivity are omitted). 
}
\end{figure}
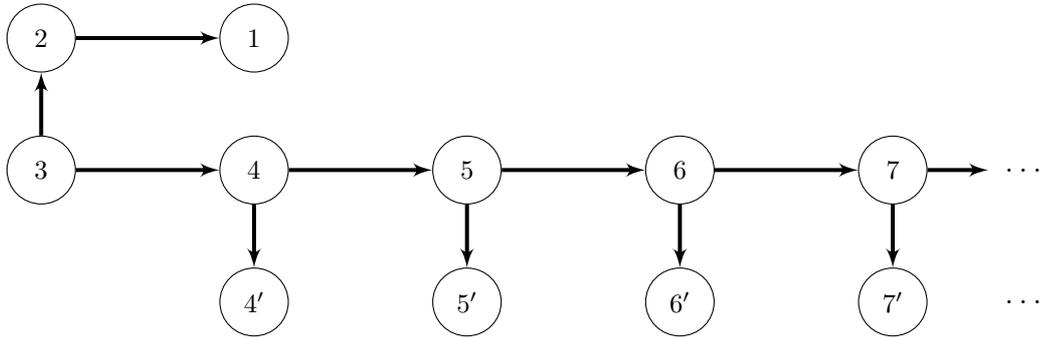


When combined, Theorem~\ref{th:equivalence} and Proposition~\ref{pro:cardinality} allow us to identify the distinct elements in $\text{Ext}(\mathcal{A})$.
\begin{proposition}\label{pro:extAeq} If $\mathcal{A}$ is regular, then
\begin{equation}\text{Ext}(\mathcal{A})= \left\{\vc q\left(\textstyle\bigcup_{i\in (I\cup J)} A_i \right): (I,J)\in \textswab{I}_\mathcal{A} \right\},\end{equation}
and distinct elements in $\textswab{I}_\mathcal{A}$ correspond to distinct extinction probability vectors.
\end{proposition}
In the example considered in Figure \ref{pic2}, the distinct elements of $\text{Ext}(\mathcal{A})$ are therefore
\begin{eqnarray}\nonumber\lefteqn{\text{Ext}(\mathcal{A})}\\&=&\nonumber \left\{\vc q\left(\textstyle\bigcup_{i\in I} A_i \right): I\in \textswab{P}_\mathcal{A}\right\}\\\nonumber&\bigcup& \left\{ \vc q\left(A_3 \cup A_4\cup A_5\cup\ldots \right), \vc q\left(A_1\cup A_3 \cup A_4\cup A_5\cup\ldots \right), \vc q\left(A_2\cup A_3 \cup A_4\cup A_5\cup\ldots \right) \right\}.\\\label{ex2ext}&&\end{eqnarray}

\subsection{The number of distinct elements in $\text{Ext}(\mathcal{A})$}
\label{subsec:cardinality}


Building on the results in the previous section, we are now ready to discuss the number of distinct elements in $\text{Ext}(\mathcal{A})$, $|\text{Ext}(\mathcal{A})|$. 
In particular, Propositions \ref{pro:cardinality} and  \ref{pro:extAeq} lead to equivalent conditions for the number of distinct elements in $\text{Ext}(\mathcal{A})$ to be finite, countably infinite, or  uncountable.


\begin{theorem}\label{th:finiteinfinite2}
Given a family $\mathcal{A}$ satisfying $(C3)$,
 \begin{enumerate}[(i)]
   \item $|\text{\text{Ext}}(\mathcal{A})|$ is finite if and only if $\mathcal{A}$ is finite (that is, $\kappa_{\mathcal{A}}<\infty$).     \end{enumerate}
If (C2), (C3) and (C5) hold then
     \begin{enumerate}[(i)]
     \setcounter{enumi}{1}
   \item  \begin{equation}\label{bound}|  \text{\text{Ext}}(\mathcal{A}) | \geq | \textswab{P}_\mathcal{A}|.\end{equation}
If, in addition, 
$\mathcal{A}$ is regular and $\textswab{C}_\mathcal{A}=\{\emptyset\}$, then there is equality in \eqref{bound}. 
\item If \text{Ext}$(\mathcal{A})$ is countably infinite, then there exists a family $\mathcal{A}'\subseteq \mathcal{A}$ satisfying (C2)-(C3)-(C5) with $\kappa_{\mathcal{A}'}=\infty$ 
such that either $A'_1 \Rightarrow A'_2 \Rightarrow A'_3 \Rightarrow \dots $ or $A'_1 \Leftarrow A'_2 \Leftarrow A'_3 \Leftarrow \dots $.
In particular if $\mathcal{A}$ is regular, one can choose $\mathcal{A}'$ as a  regular family.
 \end{enumerate}
 
 If $\mathcal{A}$ is regular,
 then
 \begin{enumerate}[(i)]
   \setcounter{enumi}{3}
   \item 
   $\text{Ext}(\mathcal{A})$ is countably infinite if and only if  $\textswab{P}_\mathcal{A}$ and $\textswab{C}_\mathcal{A}$ are both countable and at least one of them is countably infinite.
   \item
   $\text{Ext}(\mathcal{A})$ 
   is uncountable if and only if either $\textswab{P}_\mathcal{A}$ is uncountable or $\textswab{C}_\mathcal{A}$ is uncountable.
 \end{enumerate}

\end{theorem}
Note that if $\mathcal{A}$ is regular, the condition `$\textswab{C}_\mathcal{A}=\{\emptyset\}$' is sufficient but not necessary for the equality in \eqref{bound} to hold. Indeed, in the example considered in Figure \ref{pic2}, $\textswab{C}_\mathcal{A}\neq \{\emptyset\}$ while  $\text{Ext}(\mathcal{A})$ and $\textswab{P}_\mathcal{A}$ are both countably infinite (see Equation \eqref{ex2ext}).

The next corollary gives a sufficient condition for the existence of an infinite regular family whose associated graph is edgeless and, as a consequence, for the existence of uncountably many distinct extinction probability vectors.

\begin{corollary}\label{th:2}
If there exists a (infinite) collection of pairwise disjoint subsets $A_1, A_2, \dots$ of $\mathcal{X}$ such that for each $i \geq 1$ there exists $x_i \in \mathcal{X}$ with
\[
\mbP_{x_i}( \mathcal{S}(A_i) \cap \mathcal{E}(\mathcal{X} \backslash A_i ) )>0,
\]then there are uncountably many distinct extinction probability vectors.
\end{corollary}

\section{Examples}\label{sec:Ex}

We are ready to answer two important questions:\begin{enumerate} 
\item The first question was asked previously in \cite{cf:BZ2020}: Is it possible to construct an irreducible MGWBP with countably infinitely many extinction probability vectors? 
Theorem~\ref{th:finiteinfinite2} not only suggests that the answer is positive, it also gives insight into how such examples may arise.
In Example 1 we not only answer this question but we go further by constructing an irreducible family of processes where, by varying a single parameter, we can transition smoothly between cases where the process has any finite number of extinction probability vectors, a countably infinite number of extinction probability vectors, and an uncountably infinite number of extinction probability vectors. 

\item Given a regular family $\mathcal{A}$, do we always have $|\text{Ext}(\mathcal{A})|=|\textswab{P}_\mathcal{A}|$? If $\textswab{P}_\mathcal{A}$ is either finite or uncountable, then equality holds. Thus, by Theorem ~\ref{th:finiteinfinite2}\emph{(v)}, we may only have $|\text{Ext}(\mathcal{A})|>|\textswab{P}_\mathcal{A}|$ if $\textswab{P}_\mathcal{A}$ is countable and $\textswab{C}_\mathcal{A}$ is uncountable.   In Example 2, both $\textswab{P}_\mathcal{A}$ and $\textswab{C}_\mathcal{A}$ are countable, and thus $|\text{Ext}(\mathcal{A})|=|\textswab{P}_\mathcal{A}|$, while in Example 3, $\textswab{P}_\mathcal{A}$ is countable and $\textswab{C}_\mathcal{A}$ is uncountable, and thus $|\text{Ext}(\mathcal{A})|>|\textswab{P}_\mathcal{A}|$. This means the answer to the above question is negative. 
\end{enumerate}
Example 1 is an application of the results developed in Section \ref{sec:comparison} and \ref{sec:extprobvec}, and Examples 2 and 3 highlight the framework developed in Section \ref{sec:extprobvec}.

%
%
%

\bigskip
\paragraph{\textbf{Example 1: From finitely many to uncountably many extinction probability vectors.}}

Consider a process with type set $\mathcal{X}= \mathbb{N}_0^2$, where \begin{itemize}
\item individuals of type $(0,0)$ have one child of type $(1,0)$ with probability $q$, and 0 children otherwise;
\item individuals of type $(0,j)$, $j \geq 1$, have one child of type $(0,j-1)$ with probability $p<1$, and 0 children otherwise;
\item individuals of type $(i,0)$, $i \geq 1$, have one child of type $(i,1)$ with probability 1, and one child of type $(i+1,0)$ with probability $q$; and
\item individuals of type $(i,j)$, $i,j \geq 1$, have a geometric number of children of type $(i-1,j)$ with mean $r^{-j+1}$, and one child of type $(i,j+1)$ with probability 1.
\end{itemize}
A visual representation of these offspring distributions is given in Figure \ref{Ex2}.
We partition $\mathcal{X}$ in two ways: by \emph{levels}, $\mathcal{L}_i := \{ (i,j) \}_{ j \geq 0}$ for $i\geq 0$,
and by \emph{phases} $\mathcal{P}_j:= \{ (i,j) \}_{ i \geq 0}$, for $j\geq 0$.

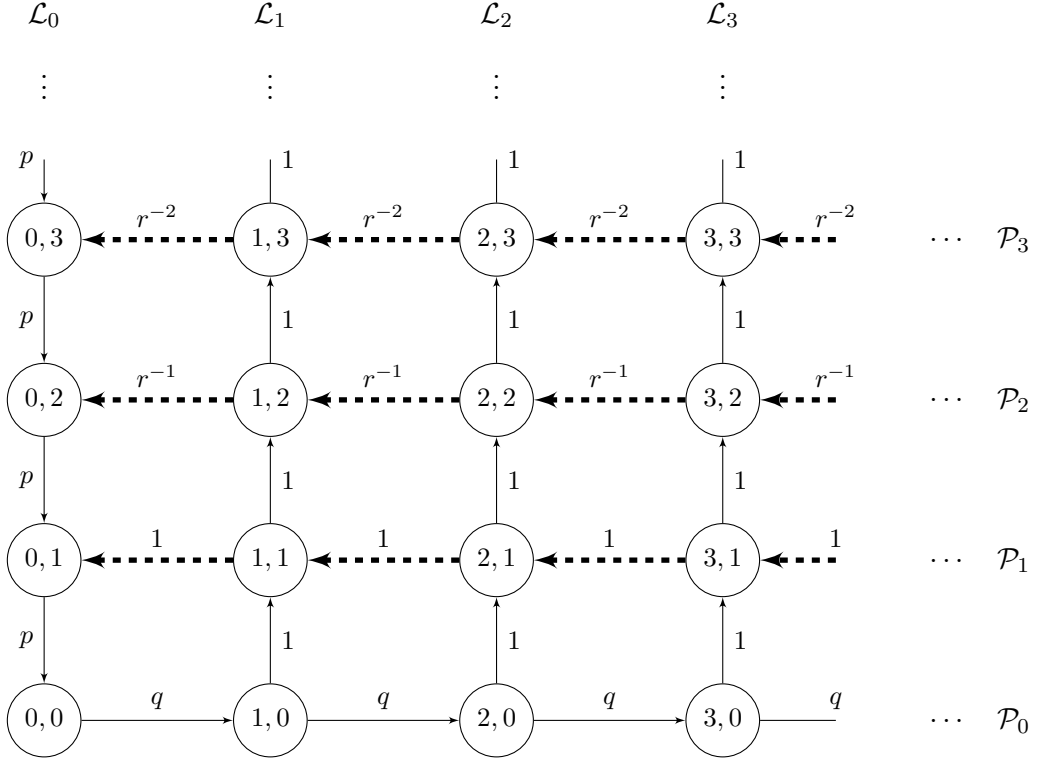
\begin{figure}[ht]
\centering
\begin{tikzpicture}[scale=0.85]

\tikzset{vertex/.style = {shape=circle,draw,minimum size=1.3em}}
\tikzset{edge/.style = {->,> = latex'}}

%

\node[vertex, minimum size=.9cm] (01) at  (0,0) {\footnotesize$0 ,0$};
\node[vertex, minimum size=.9cm] (11) at  (3.5,0) {\footnotesize $1, 0 $};
\node[vertex, minimum size=.9cm] (21) at  (7,0) {\footnotesize $2,0$};
\node[vertex, minimum size=.9cm] (31) at  (10.5,0) {\footnotesize $3,0$};

\node[vertex, minimum size=.9cm] (02) at  (0,2.5) {\footnotesize $ 0, 1 $};
\node[vertex, minimum size=.9cm] (12) at  (3.5,2.5) {\footnotesize $1,1$};
\node[vertex, minimum size=.9cm] (22) at  (7,2.5) {\footnotesize $2,1$};
\node[vertex, minimum size=.9cm] (32) at  (10.5,2.5) {\footnotesize $3,1$};

\node[vertex, minimum size=.9cm] (03) at  (0,5) {\footnotesize $ 0, 2 $};
\node[vertex, minimum size=.9cm] (13) at  (3.5,5) {\footnotesize $1,2$};
\node[vertex, minimum size=.9cm] (23) at  (7,5) {\footnotesize $2,2$};
\node[vertex, minimum size=.9cm] (33) at  (10.5,5) {\footnotesize $3,2$};

\node[vertex, minimum size=.9cm] (04) at  (0,7.5) {\footnotesize $ 0, 3 $};
\node[vertex, minimum size=.9cm] (14) at  (3.5,7.5) {\footnotesize $1,3$};
\node[vertex, minimum size=.9cm] (24) at  (7,7.5) {\footnotesize $2,3$};
\node[vertex, minimum size=.9cm] (34) at  (10.5,7.5) {\footnotesize $3,3$};

\node at (14,7.5) {\small $\dots$};
\node at (14,5) {\small $\dots$};
\node at (14,2.5) {\small $\dots$};
\node at (14,0) {\small $\dots$};

\node at (15,7.5) {\small $\mathcal{P}_3$};
\node at (15,5) {\small $\mathcal{P}_2$};
\node at (15,2.5) {\small $\mathcal{P}_1$};
\node at (15,0) {\small $\mathcal{P}_0$};

\node at (0,10) {\small $\vdots$};
\node at (3.5,10) {\small $\vdots$};
\node at (7,10) {\small $\vdots$};
\node at (10.5,10) {\small $\vdots$};

\node at (0,11) {\small $\mathcal{L}_0$};
\node at (3.5,11) {\small $\mathcal{L}_1$};
\node at (7,11) {\small $\mathcal{L}_2$};
\node at (10.5,11) {\small $\mathcal{L}_3$};



%

\draw[edge,above] (01) to[bend left=0] node {\footnotesize $q$ } (11);
\draw[edge,above] (11) to[bend left=0] node {\footnotesize $q $ } (21);
\draw[edge,above] (21) to[bend left=0] node {\footnotesize $q $ } (31);

\draw[edge,left] (02) to[bend right=0] node {\footnotesize $p$ } (01);
\draw[edge,left] (03) to[bend right=0] node {\footnotesize $p$ } (02);
\draw[edge,left] (04) to[bend right=0] node {\footnotesize $p$ } (03);


\draw[edge,right] (11) to[bend left=0] node {\footnotesize $1$ } (12);
\draw[edge,right] (12) to[bend left=0] node {\footnotesize $1$ } (13);
\draw[edge,right] (13) to[bend left=0] node {\footnotesize $1$ } (14);

\draw[edge,right] (21) to[bend left=0] node {\footnotesize $1$ } (22);
\draw[edge,right] (22) to[bend left=0] node {\footnotesize $1$ } (23);
\draw[edge,right] (23) to[bend left=0] node {\footnotesize $1$ } (24);

\draw[edge,right] (31) to[bend left=0] node {\footnotesize $1$ } (32);
\draw[edge,right] (32) to[bend left=0] node {\footnotesize $1$ } (33);
\draw[edge,right] (33) to[bend left=0] node {\footnotesize $1$ } (34);

\draw[edge,above,  dashed, line width=2pt] (32) to[bend left=0] node {\footnotesize $1$ } (22);
\draw[edge,above,  dashed, line width=2pt] (22) to[bend left=0] node {\footnotesize $1$ } (12);
\draw[edge,above,  dashed, line width=2pt] (12) to[bend left=0] node {\footnotesize $1$ } (02);

\draw[edge,above,  dashed, line width=2pt] (33) to[bend left=0] node {\footnotesize $r^{-1}$ } (23);
\draw[edge,above,  dashed, line width=2pt] (23) to[bend left=0] node {\footnotesize $r^{-1}$ } (13);
\draw[edge,above,  dashed, line width=2pt] (13) to[bend left=0] node {\footnotesize $r^{-1}$ } (03);

\draw[edge,above, dashed, line width=2pt] (34) to[bend left=0] node {\footnotesize $r^{-2}$ } (24);
\draw[edge,above,  dashed, line width=2pt] (24) to[bend left=0] node {\footnotesize $r^{-2}$ } (14);
\draw[edge,above,  dashed, line width=2pt] (14) to[bend left=0] node {\footnotesize $r^{-2}$ } (04);

\draw[above] (31) to[pos=1] node {\footnotesize $q$} (12.25,0);

\draw[edge,above,  dashed, line width=2pt] (12.25,2.5) to[pos=0] node {\footnotesize $1$} (32);
\draw[edge,above,  dashed, line width=2pt] (12.25,5) to[pos=0] node {\footnotesize $r^{-1}$} (33);
\draw[edge,above,  dashed, line width=2pt] (12.25,7.5) to[pos=0] node {\footnotesize $r^{-2}$} (34);

\draw[right] (34) to[pos=1] node {\footnotesize $1$} (10.5,8.75);
\draw[right] (24) to[pos=1] node {\footnotesize $1$} (7,8.75);
\draw[right] (14) to[pos=1] node {\footnotesize $1$} (3.5,8.75);

\draw[edge,left] (-0,8.75) to[bend right=0,pos=0] node {\footnotesize $p$}  (04);

\end{tikzpicture}
\caption{\label{Ex2} A visual representation of the offspring distributions in Example 1. The solid arrows represent Bernoulli distributions and bold dashed arrows represent geometric distributions 
(the weights represent the corresponding means).
}
\end{figure}

Consider the family
$\mathcal{A}=\{\mathcal{L}_1,\mathcal{L}_2,\ldots\}$. 
The next proposition implies that, for any $p,q<1$, we can choose $r$ such that the process has any finite number $k \geq 1$ of extinction probability vectors ($p^{1/(k-1)}<r\leq p^{1/k}$), which corresponds to
$$\mathcal{L}_1\Leftarrow  \mathcal{L}_2 \Leftarrow\ldots \Leftarrow \mathcal{L}_{k-1}\Leftarrow \mathcal{L}_k\Leftrightarrow \mathcal{L}_{k+1}\Leftrightarrow\ldots,$$
countably infinite many distinct extinction probability vectors ($r=1$), which corresponds to
$$\mathcal{L}_1\Leftarrow  \mathcal{L}_2 \Leftarrow  \mathcal{L}_3\Leftarrow  \mathcal{L}_4 \Leftarrow\ldots,$$
or uncountably many distinct extinction probability vectors ($r>1$), which corresponds to
$$\mathcal{L}_1\nLeftrightarrow  \mathcal{L}_2 \nLeftrightarrow  \mathcal{L}_3\nLeftrightarrow \mathcal{L}_4 \nLeftrightarrow\ldots.$$
Moreover, the proposition implies that, when $r\leq 1$, Ext $=$ Ext$(\mathcal{A})$. Note that in this example, $\textswab{C}_{\mathcal A}=\emptyset$, and when $r\leq 1$, the only primitive subsets are singletons.
In preparation for the next result, for any $A\subseteq\mathcal{X}$ we let 
\[
\iota(A) := \min\{ i\geq 0 \colon |\mathcal{L}_i \cap A|= \infty \}, 
\]
and set $\iota(A):=\infty$ if the above set is empty.

\begin{proposition}\label{pro:ultimate}
In Example 1, 
\begin{itemize}
\item[(i)] if $r<1$, then there is a finite number $i^*:=\min\{i\geq 1\colon r^i\leq p\}$ of distinct extinction probability vectors, namely $\vc q=\bm{\tilde{q}}$ if $i^*=1$, and
 \begin{equation}\label{eq2}\vc q=\bm{q}(\mathcal{L}_1) <\ldots < \bm{q}(\mathcal{L}_{i^*})= \bm{\tilde{q}}\qquad\textrm{ if $i^*\geq2$.}\end{equation} In particular, if $\iota(A)<i^*$ then $\bm{q}(A)=\bm{q}(\mathcal{L}_{\iota(A)})$, whereas if $\iota(A)\geq i^*$ then $\bm{q}(A)=\bm{\tilde q}$.
\item[(ii)] if $r=1$, then there are countably infinite many distinct extinction probability vectors, namely
\begin{equation}\label{eq1}\vc q=\bm{q}(\mathcal{L}_1) < \bm{q}(\mathcal{L}_{2})<\bm{q}(\mathcal{L}_3)<\ldots,\end{equation}and $\tilde{\vc q}$. In particular, if $\iota(A)<\infty$ then $\bm{q}(A)=\bm{q}(\mathcal{L}_{\iota(A)})$, whereas if $\iota(A)=\infty$ then $\bm{q}(A)=\bm{\tilde q}$.
 \item[(iii)] if $r> 1$, then there are uncountably many distinct extinction probability vectors.
\end{itemize}
\end{proposition}

\begin{figure}[ht]
\includegraphics[scale=0.35]{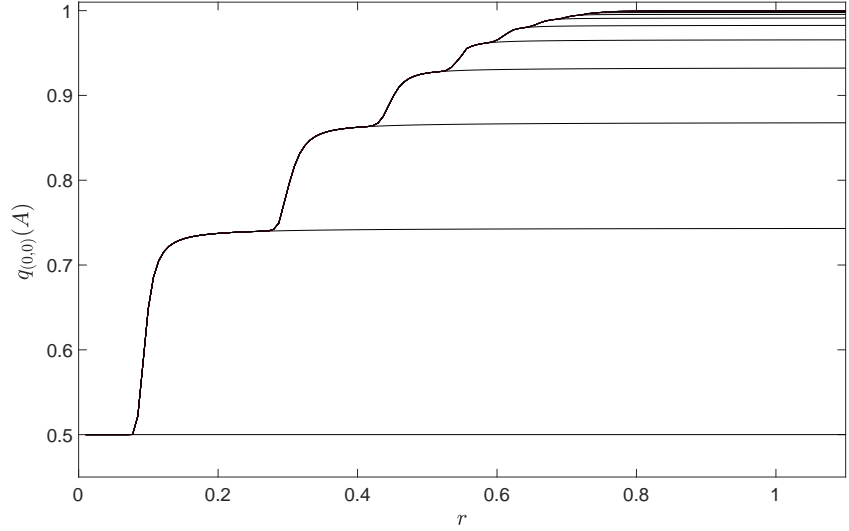}
\caption{\label{f1}
The probabilities of extinction $q_{(0,0)}(\mathcal{L}_1)$ (lowest curve), $q_{(0,0)}(\mathcal{L}_2)$ (second lowest curve), $q_{(0,0)}(\mathcal{L}_3)$, \ldots, as a function of $r$ when $p=0.1$ and $q=0.5$. 
}\end{figure}

 
 Figure \ref{f1} shows the distinct probabilities of extinction 
 $\{q_{(0,0)}(\mathcal{L}_i)\}_{i \ge 1}$
as a function of $r$ when $p=0.1$ and $q=0.5$.  
Observe that, in accordance with Proposition~\ref{pro:ultimate}, the number of extinction probabilities increases by one at $r=\sqrt[i]{p}$ for each $i \ge 1$.
 The probabilities are computed using the iterative method presented in Appendix A.

We now consider what may happen if the family $\mathcal{A}$ is not chosen carefully (i.e. is not regular).
Consider the family $\mathcal{A}' = \{ \mathcal L'_0,\mathcal L'_1,\mathcal L'_2,\ldots \}$, where
\[
\mathcal L'_{i} = \big( \bigcup_{k=0}^\infty \big \{(i, 2k) \big\} \big) \cup \big( \bigcup_{k=0}^\infty  \big \{(k, 2i+1) \big\} \big),\qquad i \geq0.
\]
Note that $\mathcal{A}'$ does not satisfy (C5): indeed we  have that $\mathcal L'_{1} \Rightarrow \bigcup_{j\in J_1} \mathcal L'_{j} $, where $J_1=\{0,2,3,4,\ldots\}$. 
The next proposition implies that, when $r>1$, $\textswab{P}_\mathcal{A'}$ is uncountable, while \emph{Ext}$(\mathcal{A}')$ is countable; this shows that, without (C5),  Theorem~\ref{th:finiteinfinite2}\emph{(ii)} might not hold.

\begin{proposition}\label{prop:f}
If $r>1$, then $\mathcal L'_i \nLeftrightarrow \mathcal L'_j$ for all $i \neq j$ and \emph{Ext}$(\mathcal{A}')$ is countably infinite. 
\end{proposition}


\bigskip
\paragraph{\textbf{Example 2: A BRW on a grid.}}



Consider a branching process with typeset $\mathbb{N} \times \mathbb{N}$ in which the  generating function of type $(i,j)$ is 
\[
G_{(i,j)}(\bm{s}) = \frac{1}{3}+ \frac{1}{2}s_{(i,j)}^3 + \frac{1}{12}s_{(i,j+1)}^3+\frac{1}{12}s_{(i+1,j)}^3.
\]
In other words, an individual of type $(i,j)$ has no children with probability $1/3$, three children of type $(i,j)$ with probability $1/2$, three children of type $(i,j+1)$ with probability $1/12$, and three children of type $(i+1,j)$ with probability $1/12$.

Suppose we would like to determine the distinct elements of Ext.
We consider the family $\mathcal{A}=\mathcal{X}$ (the set of singletons), in which 
\[
(i_1, j_1) \Rightarrow (i_2, j_2) \quad \text{ if and only if } \quad i_1 \leq i_2  \text{ and } j_1 \leq j_2,
\]and whose associated graph $G_{\mathcal{A}}$ is illustrated in Figure \ref{fig:Ex3M} (the edges implied by transitivity are omitted). Note that the family $\mathcal{A}$ is regular; indeed, $(C_1)$ and $(C_3)$ are immediate, $(C_4)$ and $(C_5)$ can be verified easily (for instance by inspecting the graph $G_{\mathcal{A}}$), and $(C_2)$ follows from the fact that the mean number of type-$(i,j)$ offspring of a type-$(i,j)$ parent is $3/2>1$.

\begin{figure}

\begin{tikzpicture}[scale=.5]


\tikzset{vertex/.style = {shape=circle, draw, line width=0.5mm, minimum size=1pt, inner sep=2pt}}

\tikzset{edge/.style = {->,> = latex'}}






%




\node[vertex, minimum size=.05cm] (00) at  (0,0){};
\node[vertex, minimum size=.05cm] (01) at  (0,2){};
\node[vertex, minimum size=.05cm] (02) at  (0,4){};
\node[vertex, minimum size=.05cm] (03) at  (0,6){};
\node[vertex, minimum size=.05cm] (04) at  (0,8){};
\node[vertex, minimum size=.05cm] (05) at  (0,10){};

\node[vertex, minimum size=.05cm] (10) at  (2,0){};
\node[vertex, minimum size=.05cm] (11) at  (2,2){};
\node[vertex, minimum size=.05cm] (12) at  (2,4){};
\node[vertex, minimum size=.05cm] (13) at  (2,6){};
\node[vertex, minimum size=.05cm] (14) at  (2,8){};
\node[vertex, minimum size=.05cm] (15) at  (2,10){};

\node[vertex, minimum size=.05cm] (20) at  (4,0){};
\node[vertex, minimum size=.05cm] (21) at  (4,2){};
\node[vertex, minimum size=.05cm] (22) at  (4,4){};
\node[vertex, minimum size=.05cm] (23) at  (4,6){};
\node[vertex, minimum size=.05cm] (24) at  (4,8){};
\node[vertex, minimum size=.05cm] (25) at  (4,10){};

\node[vertex, minimum size=.05cm] (30) at  (6,0){};
\node[vertex, minimum size=.05cm] (31) at  (6,2){};
\node[vertex, minimum size=.05cm] (32) at  (6,4){};
\node[vertex, minimum size=.05cm] (33) at  (6,6){};
\node[vertex, minimum size=.05cm] (34) at  (6,8){};
\node[vertex, minimum size=.05cm] (35) at  (6,10){};

\node[vertex, minimum size=.05cm] (40) at  (8,0){};
\node[vertex, minimum size=.05cm] (41) at  (8,2){};
\node[vertex, minimum size=.05cm] (42) at  (8,4){};
\node[vertex, minimum size=.05cm] (43) at  (8,6){};
\node[vertex, minimum size=.05cm] (44) at  (8,8){};
\node[vertex, minimum size=.05cm] (45) at  (8,10){};

\node[vertex, minimum size=.05cm] (50) at  (10,0){};
\node[vertex, minimum size=.05cm] (51) at  (10,2){};
\node[vertex, minimum size=.05cm] (52) at  (10,4){};
\node[vertex, minimum size=.05cm] (53) at  (10,6){};
\node[vertex, minimum size=.05cm] (54) at  (10,8){};
\node[vertex, minimum size=.05cm] (55) at  (10,10){};

\draw[edge, line width=0.3mm] (00) to (01);
\draw[edge, line width=0.3mm] (01) to (02);
\draw[edge, line width=0.3mm] (02) to (03);
\draw[edge, line width=0.3mm] (03) to (04);
\draw[edge, line width=0.3mm] (04) to (05);

\draw[edge, line width=0.3mm] (10) to (11);
\draw[edge, line width=0.3mm] (11) to (12);
\draw[edge, line width=0.3mm] (12) to (13);
\draw[edge, line width=0.3mm] (13) to (14);
\draw[edge, line width=0.3mm] (14) to (15);

\draw[edge, line width=0.3mm] (20) to (21);
\draw[edge, line width=0.3mm] (21) to (22);
\draw[edge, line width=0.3mm] (22) to (23);
\draw[edge, line width=0.3mm] (23) to (24);
\draw[edge, line width=0.3mm] (24) to (25);

\draw[edge, line width=0.3mm] (30) to (31);
\draw[edge, line width=0.3mm] (31) to (32);
\draw[edge, line width=0.3mm] (32) to (33);
\draw[edge, line width=0.3mm] (33) to (34);
\draw[edge, line width=0.3mm] (34) to (35);

\draw[edge, line width=0.3mm] (40) to (41);
\draw[edge, line width=0.3mm] (41) to (42);
\draw[edge, line width=0.3mm] (42) to (43);
\draw[edge, line width=0.3mm] (43) to (44);
\draw[edge, line width=0.3mm] (44) to (45);

\draw[edge, line width=0.3mm] (50) to (51);
\draw[edge, line width=0.3mm] (51) to (52);
\draw[edge, line width=0.3mm] (52) to (53);
\draw[edge, line width=0.3mm] (53) to (54);
\draw[edge, line width=0.3mm] (54) to (55);

\draw[edge, line width=0.3mm] (00) to (10);
\draw[edge, line width=0.3mm] (10) to (20);
\draw[edge, line width=0.3mm] (20) to (30);
\draw[edge, line width=0.3mm] (30) to (40);
\draw[edge, line width=0.3mm] (40) to (50);

\draw[edge, line width=0.3mm] (01) to (11);
\draw[edge, line width=0.3mm] (11) to (21);
\draw[edge, line width=0.3mm] (21) to (31);
\draw[edge, line width=0.3mm] (31) to (41);
\draw[edge, line width=0.3mm] (41) to (51);

\draw[edge, line width=0.3mm] (02) to (12);
\draw[edge, line width=0.3mm] (12) to (22);
\draw[edge, line width=0.3mm] (22) to (32);
\draw[edge, line width=0.3mm] (32) to (42);
\draw[edge, line width=0.3mm] (42) to (52);

\draw[edge, line width=0.3mm] (03) to (13);
\draw[edge, line width=0.3mm] (13) to (23);
\draw[edge, line width=0.3mm] (23) to (33);
\draw[edge, line width=0.3mm] (33) to (43);
\draw[edge, line width=0.3mm] (43) to (53);

\draw[edge, line width=0.3mm] (04) to (14);
\draw[edge, line width=0.3mm] (14) to (24);
\draw[edge, line width=0.3mm] (24) to (34);
\draw[edge, line width=0.3mm] (34) to (44);
\draw[edge, line width=0.3mm] (44) to (54);

\draw[edge, line width=0.3mm] (05) to (15);
\draw[edge, line width=0.3mm] (15) to (25);
\draw[edge, line width=0.3mm] (25) to (35);
\draw[edge, line width=0.3mm] (35) to (45);
\draw[edge, line width=0.3mm] (45) to (55);

\draw[line width=0.3mm] (50) to (11,0);
\draw[line width=0.3mm] (51) to (11,2);
\draw[line width=0.3mm] (52) to (11,4);
\draw[line width=0.3mm] (53) to (11,6);
\draw[line width=0.3mm] (54) to (11,8);
\draw[line width=0.3mm] (55) to (11,10);

\draw[line width=0.3mm] (05) to (0,11);
\draw[line width=0.3mm] (15) to (2,11);
\draw[line width=0.3mm] (25) to (4,11);
\draw[line width=0.3mm] (35) to (6,11);
\draw[line width=0.3mm] (45) to (8,11);
\draw[line width=0.3mm] (55) to (10,11);

\end{tikzpicture} \hspace{1cm} 
\begin{tikzpicture}[scale=.5]


\tikzset{vertex/.style = {shape=circle, draw, line width=0.5mm, minimum size=1pt, inner sep=2pt}}

\tikzset{edge/.style = {->,> = latex'}}






%




\node[green, fill=green, vertex, minimum size=.05cm] (00) at  (0,0){};
\node[green, fill=green, vertex, minimum size=.05cm] (01) at  (0,2){};
\node[red, fill=red,vertex, minimum size=.05cm] (02) at  (0,4){};
\node[vertex, minimum size=.05cm] (03) at  (0,6){};
\node[vertex, minimum size=.05cm] (04) at  (0,8){};
\node[blue, fill=blue, vertex, minimum size=.05cm] (05) at  (0,10){};

\node[green, fill=green, vertex, minimum size=.05cm] (10) at  (2,0){};
\node[green, fill=green, vertex, minimum size=.05cm] (11) at  (2,2){};
\node[red, fill=red,vertex, minimum size=.05cm] (12) at  (2,4){};
\node[blue, fill=blue,vertex, minimum size=.05cm] (13) at  (2,6){};
\node[vertex, minimum size=.05cm] (14) at  (2,8){};
\node[vertex, minimum size=.05cm] (15) at  (2,10){};

\node[green, fill=green, vertex, minimum size=.05cm] (20) at  (4,0){};
\node[green, fill=green, vertex, minimum size=.05cm] (21) at  (4,2){};
\node[red, fill=red,vertex, minimum size=.05cm] (22) at  (4,4){};
\node[red, fill=red,vertex, minimum size=.05cm] (23) at  (4,6){};
\node[vertex, minimum size=.05cm] (24) at  (4,8){};
\node[vertex, minimum size=.05cm] (25) at  (4,10){};

\node[green, fill=green, vertex, minimum size=.05cm] (30) at  (6,0){};
\node[green, fill=green, vertex, minimum size=.05cm] (31) at  (6,2){};
\node[vertex, minimum size=.05cm] (32) at  (6,4){};
\node[red, fill=red,vertex, minimum size=.05cm] (33) at  (6,6){};
\node[vertex, minimum size=.05cm] (34) at  (6,8){};
\node[vertex, minimum size=.05cm] (35) at  (6,10){};

\node[green, fill=green, vertex, minimum size=.05cm] (40) at  (8,0){};
\node[green, fill=green, vertex, minimum size=.05cm] (41) at  (8,2){};
\node[blue, fill=blue,vertex, minimum size=.05cm] (42) at  (8,4){};
\node[red, fill=red,vertex, minimum size=.05cm] (43) at  (8,6){};
\node[red, fill=red,vertex, minimum size=.05cm] (44) at  (8,8){};
\node[vertex, minimum size=.05cm] (45) at  (8,10){};

\node[green, fill=green, vertex, minimum size=.05cm] (50) at  (10,0){};
\node[green, fill=green, vertex, minimum size=.05cm] (51) at  (10,2){};
\node[vertex, minimum size=.05cm] (52) at  (10,4){};
\node[vertex, minimum size=.05cm] (53) at  (10,6){};
\node[red, fill=red,vertex, minimum size=.05cm] (54) at  (10,8){};
\node[vertex, minimum size=.05cm] (55) at  (10,10){};

\draw[edge, line width=0.3mm] (00) to (01);
\draw[edge, line width=0.3mm] (01) to (02);
\draw[edge, line width=0.3mm] (02) to (03);
\draw[edge, line width=0.3mm] (03) to (04);
\draw[edge, line width=0.3mm] (04) to (05);

\draw[edge, line width=0.3mm] (10) to (11);
\draw[edge, line width=0.3mm] (11) to (12);
\draw[edge, line width=0.3mm] (12) to (13);
\draw[edge, line width=0.3mm] (13) to (14);
\draw[edge, line width=0.3mm] (14) to (15);

\draw[edge, line width=0.3mm] (20) to (21);
\draw[edge, line width=0.3mm] (21) to (22);
\draw[edge, line width=0.3mm] (22) to (23);
\draw[edge, line width=0.3mm] (23) to (24);
\draw[edge, line width=0.3mm] (24) to (25);

\draw[edge, line width=0.3mm] (30) to (31);
\draw[edge, line width=0.3mm] (31) to (32);
\draw[edge, line width=0.3mm] (32) to (33);
\draw[edge, line width=0.3mm] (33) to (34);
\draw[edge, line width=0.3mm] (34) to (35);

\draw[edge, line width=0.3mm] (40) to (41);
\draw[edge, line width=0.3mm] (41) to (42);
\draw[edge, line width=0.3mm] (42) to (43);
\draw[edge, line width=0.3mm] (43) to (44);
\draw[edge, line width=0.3mm] (44) to (45);

\draw[edge, line width=0.3mm] (50) to (51);
\draw[edge, line width=0.3mm] (51) to (52);
\draw[edge, line width=0.3mm] (52) to (53);
\draw[edge, line width=0.3mm] (53) to (54);
\draw[edge, line width=0.3mm] (54) to (55);

\draw[edge, line width=0.3mm] (00) to (10);
\draw[edge, line width=0.3mm] (10) to (20);
\draw[edge, line width=0.3mm] (20) to (30);
\draw[edge, line width=0.3mm] (30) to (40);
\draw[edge, line width=0.3mm] (40) to (50);

\draw[edge, line width=0.3mm] (01) to (11);
\draw[edge, line width=0.3mm] (11) to (21);
\draw[edge, line width=0.3mm] (21) to (31);
\draw[edge, line width=0.3mm] (31) to (41);
\draw[edge, line width=0.3mm] (41) to (51);

\draw[edge, line width=0.3mm] (02) to (12);
\draw[edge, line width=0.3mm] (12) to (22);
\draw[edge, line width=0.3mm] (22) to (32);
\draw[edge, line width=0.3mm] (32) to (42);
\draw[edge, line width=0.3mm] (42) to (52);

\draw[edge, line width=0.3mm] (03) to (13);
\draw[edge, line width=0.3mm] (13) to (23);
\draw[edge, line width=0.3mm] (23) to (33);
\draw[edge, line width=0.3mm] (33) to (43);
\draw[edge, line width=0.3mm] (43) to (53);

\draw[edge, line width=0.3mm] (04) to (14);
\draw[edge, line width=0.3mm] (14) to (24);
\draw[edge, line width=0.3mm] (24) to (34);
\draw[edge, line width=0.3mm] (34) to (44);
\draw[edge, line width=0.3mm] (44) to (54);

\draw[edge, line width=0.3mm] (05) to (15);
\draw[edge, line width=0.3mm] (15) to (25);
\draw[edge, line width=0.3mm] (25) to (35);
\draw[edge, line width=0.3mm] (35) to (45);
\draw[edge, line width=0.3mm] (45) to (55);

\draw[line width=0.3mm] (50) to (11,0);
\draw[line width=0.3mm] (51) to (11,2);
\draw[line width=0.3mm] (52) to (11,4);
\draw[line width=0.3mm] (53) to (11,6);
\draw[line width=0.3mm] (54) to (11,8);
\draw[line width=0.3mm] (55) to (11,10);

\draw[line width=0.3mm] (05) to (0,11);
\draw[line width=0.3mm] (15) to (2,11);
\draw[line width=0.3mm] (25) to (4,11);
\draw[line width=0.3mm] (35) to (6,11);
\draw[line width=0.3mm] (45) to (8,11);
\draw[line width=0.3mm] (55) to (10,11);

\end{tikzpicture}

\caption{\label{fig:Ex3M} Left panel: The graph $G_{\mathcal{A}}$ in Example 3. Right panel: The graph $G_{\mathcal{A}}$ in which particular subsets of vertices are highlighted.}

\end{figure}
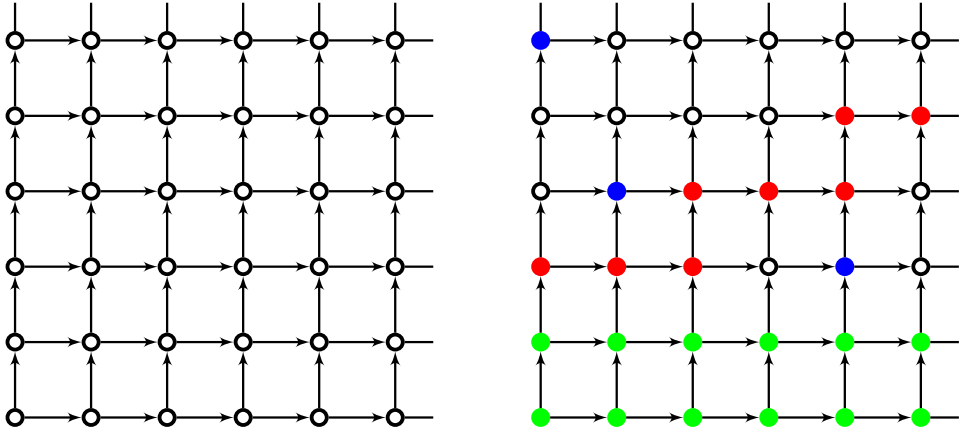

In this example, the primitive subsets are the subsets of $\mathcal{X}$ in which no element is strictly greater (componentwise) than any other. More formally, 
\[
\textswab{P}_{\mathcal{A}} = \{ A \subset \mathcal{X} : \nexists  (i_1,j_1), (i_2,j_2) \in A \text{ with }  i_1 \leq i_2  \text{ and } j_1 \leq j_2 \}. 
\]
The set of blue nodes in  Figure \ref{fig:Ex3M} is an example of a primitive subset. Note that every element of $\textswab{P}_{\mathcal{A}}$ is a finite subset, and therefore $\textswab{P}_{\mathcal{A}}$ is countable.
The set of representatives of pure ascending chains is
\begin{equation}\label{Ex3M1}
\textswab{C}_{\mathcal{A}} =\left\{ (i,j): 1\leq i\leq k, j\geq 1 \right\}_{k\in \mathbb{N}}  \cup\left\{  (i,j): i\geq 1,1\leq j\leq k\right\}_{k\in \mathbb{N}}\cup \mathcal{X} \cup \emptyset.
\end{equation}
To understand how this expression for $\textswab{C}_{\mathcal{A}}$ is obtained, observe that there are essentially three kinds of ascending chains: those that take infinitely many steps upwards while only taking finitely many steps to the right (representatives of these chains are given in the first term of \eqref{Ex3M1}), those that take only finitely many steps upwards while taking infinitely many steps to the right (representatives of these chains are given in the second term of \eqref{Ex3M1}; the set of green nodes in  Figure \ref{fig:Ex3M} corresponds to $k=2$), and those that take both infinitely many steps upwards and infinitely many steps to the right (these chains have just a single representative $\mathcal{X}$; one such path is illustrated in red in Figure \ref{fig:Ex3M}). 

By Proposition~\ref{pro:extAeq}  the set of distinct extinction probability vectors is 
$$\text{Ext}=\text{Ext}(\mathcal{A})= \left\{\vc q\left(\textstyle\bigcup_{(i,j) \in I \cup J} (i,j) \right): (I,J)\in \textswab{I}_\mathcal{A} \right\},$$
where
\begin{align*}
\textswab{J}_{\mathcal{A}} &= \{ (I,J) \in (\textswab{P}_{\mathcal{A}} \times \textswab{C}_{\mathcal{A}}): I \cap J = \emptyset, (J \backslash I^+)^+=J \} \\
&=\{ (I,J) \in (\textswab{P}_{\mathcal{A}} \times \textswab{C}_{\mathcal{A}}): I \cap J = \emptyset \},
\end{align*}
and the final equality follows from the fact that for every $I\in\textswab{P}_{\mathcal{A}}$, $I^+$ is a finite set. One element $(I,J)$ of $\textswab{J}_{\mathcal{A}}$ is formed by letting $I$ and $J$ be the set of blue and green nodes respectively in Figure \ref{fig:Ex3M}.
Because $\textswab{P}_{\mathcal{A}}$ and $\textswab{C}_{\mathcal{A}}$ are both countably infinite, by Theorem \ref{th:finiteinfinite2}, ${\rm Ext}$ contains a countably infinite number of distinct elements.
We have thus constructed an example with ascending chains in which $|\textswab{P}_{\mathcal{A}}|=|{\rm Ext}(\mathcal{A}) |$. 


\bigskip
\paragraph{\textbf{Example 3: A BRW on a modified binary tree.}}

Consider the modification of an oriented binary tree which is illustrated in Figure~\ref{fig:uncountable} and is formally constructed as follows. Let $Z:=\bigcup_{i=0}^{+\infty} \{-1,+1\}^i$ denote the set of vertices,
 where $\{-1,+1\}^0=\{\emptyset\}$
 represents the root. Note that every vertex is a finite sequence of $-1$ and $+1$. 
 A planar representation of this set is given by the map $\gamma:Z \mapsto\mathbb{R}^2$ where
 $\gamma(\{\emptyset\})=(0,0)$ and $\gamma(\{\alpha_1,\alpha_2,\ldots,\alpha_n\})=
 \big ( \sum_{i=1}^n \alpha_i 3^{-i}, n \big )$, for $n\geq 1$. 
 Henceforth, when we speak of ``left'' and ``right'' we refer to the first coordinate in this planar representation.
 Given a vertex $\{\emptyset\}$ or $\{\alpha_1,\dots,\alpha_n\}$
 with $n\geq 1$, we define the (oriented) edges as follows
 \[
 \begin{matrix}
  (\{\beta_i\}_{i=1}^m, \{\alpha_i\}_{i=1}^n)\\
    \Updownarrow\\
    \begin{cases}
     \alpha_i=\beta_i\; \forall i \le m.&\text{if } m+1=n \ge 1
     \\
 \beta_i=\alpha_i\; \forall i \le n-1,\, \beta_n=-\alpha_n=1, \,
     \beta_i=-1\; \forall i >n& \text{if } m \ge n \ge 1.
 \\
  \end{cases}
  \end{matrix}
 \]
  Roughly speaking, the first line defines the usual upward edges in the binary tree (where each parent has exactly two children). The second line 
 draws lateral edges to each point from the sibling on its right (if any) and from each descendent of this siblings in such a way that the resulting graph is isomorphic to a planar graph (see Figure~\ref{fig:uncountable}). We observe that there are no lateral edges pointing to the right, and that from every vertex $\{\beta_i\}_{i=1}^m$ such that $\beta_i=1$ for some $i$, there is always a lateral edge pointing to the left (to the sibling if $\beta_n=1$, or to the sibling of some ancestor if $\beta_n=-1$). Denote this collection of edges by $E_Z$; it is easy to see that there are no cycles. 
 
 \begin{figure}

\centering

\begin{tikzpicture}[scale=.67]


\tikzset{vertex/.style = {shape=circle, draw, line width=0.5mm, minimum size=1pt, inner sep=2pt}}

\tikzset{edge/.style = {->,> = latex'}}






%




\node[vertex, minimum size=.05cm] (01) at  (0,0){};

\node[vertex, minimum size=.1cm] (11) at  (-6,2){};

\node[vertex, minimum size=.1cm] (12) at  (6,2){};

\node[vertex, minimum size=.1cm] (21) at  (-8,4){};

\node[vertex, minimum size=.1cm] (22) at  (-4,4){};

\node[vertex, minimum size=.1cm] (23) at  (4,4){};

\node[vertex, minimum size=.1cm] (24) at  (8,4){};

\node[vertex, minimum size=.1cm] (31) at  (-9,6){};

\node[vertex, minimum size=.1cm] (32) at  (-7,6){};

\node[vertex, minimum size=.1cm] (33) at  (-5,6){};

\node[vertex, minimum size=.1cm] (34) at  (-3,6){};

\node[vertex, minimum size=.1cm] (35) at  (3,6){};

\node[vertex, minimum size=.1cm] (36) at  (5,6){};

\node[vertex, minimum size=.1cm] (37) at  (7,6){};

\node[vertex, minimum size=.1cm] (38) at  (9,6){};

\node[vertex, minimum size=.1cm] (41) at  (-9.5,8){};

\node[vertex, minimum size=.1cm] (42) at  (-8.5,8){};

\node[vertex, minimum size=.1cm] (43) at  (-7.5,8){};

\node[vertex, minimum size=.1cm] (44) at  (-6.5,8){};

\node[vertex, minimum size=.1cm] (45) at  (-5.5,8){};

\node[vertex, minimum size=.1cm] (46) at  (-4.5,8){};

\node[vertex, minimum size=.1cm] (47) at  (-3.5,8){};

\node[vertex, minimum size=.1cm] (48) at  (-2.5,8){};

\node[vertex, minimum size=.1cm] (49) at  (2.5,8){};

\node[vertex, minimum size=.1cm] (410) at  (3.5,8){};

\node[vertex, minimum size=.1cm] (411) at  (4.5,8){};

\node[vertex, minimum size=.1cm] (412) at  (5.5,8){};

\node[vertex, minimum size=.1cm] (413) at  (6.5,8){};

\node[vertex, minimum size=.1cm] (414) at  (7.5,8){};

\node[vertex, minimum size=.1cm] (415) at  (8.5,8){};

\node[vertex, minimum size=.1cm] (416) at  (9.5,8){};

\draw[line width=0.3mm] (41) to (-9.65,9);

\draw[line width=0.3mm] (41) to (-9.35,9);

\draw[line width=0.3mm] (42) to (-8.65,9);

\draw[line width=0.3mm] (42) to (-8.35,9);

\draw[line width=0.3mm] (43) to (-7.65,9);

\draw[line width=0.3mm] (43) to (-7.35,9);

\draw[line width=0.3mm] (44) to (-6.65,9);

\draw[line width=0.3mm] (44) to (-6.35,9);

\draw[line width=0.3mm] (45) to (-5.65,9);

\draw[line width=0.3mm] (45) to (-5.35,9);

\draw[line width=0.3mm] (46) to (-4.65,9);

\draw[line width=0.3mm] (46) to (-4.35,9);

\draw[line width=0.3mm] (47) to (-3.65,9);

\draw[line width=0.3mm] (47) to (-3.35,9);

\draw[line width=0.3mm] (48) to (-2.65,9);

\draw[line width=0.3mm] (48) to (-2.35,9);

\draw[line width=0.3mm] (49) to (2.65,9);

\draw[line width=0.3mm] (49) to (2.35,9);

\draw[line width=0.3mm] (410) to (3.65,9);

\draw[line width=0.3mm] (410) to (3.35,9);

\draw[line width=0.3mm] (411) to (4.65,9);

\draw[line width=0.3mm] (411) to (4.35,9);

\draw[line width=0.3mm] (412) to (5.65,9);

\draw[line width=0.3mm] (412) to (5.35,9);

\draw[line width=0.3mm] (413) to (6.65,9);

\draw[line width=0.3mm] (413) to (6.35,9);

\draw[line width=0.3mm] (414) to (7.65,9);

\draw[line width=0.3mm] (414) to (7.35,9);

\draw[line width=0.3mm] (415) to (8.65,9);

\draw[line width=0.3mm] (415) to (8.35,9);

\draw[line width=0.3mm] (416) to (9.65,9);

\draw[line width=0.3mm] (416) to (9.35,9);

\draw[edge, line width=0.3mm] (01) to (11);

\draw[edge, line width=0.3mm] (01) to (12);

\draw[edge, line width=0.3mm] (11) to (21);

\draw[edge, line width=0.3mm] (11) to (22);

\draw[edge, line width=0.3mm] (12) to (23);

\draw[edge, line width=0.3mm] (12) to (24);

\draw[edge, line width=0.3mm] (21) to (31);

\draw[edge, line width=0.3mm] (21) to (32);

\draw[edge, line width=0.3mm] (22) to (33);

\draw[edge, line width=0.3mm] (22) to (34);

\draw[edge, line width=0.3mm] (23) to (35);

\draw[edge, line width=0.3mm] (23) to (36);

\draw[edge, line width=0.3mm] (24) to (37);

\draw[edge, line width=0.3mm] (24) to (38);

\draw[edge, line width=0.3mm] (31) to (41);

\draw[edge, line width=0.3mm] (31) to (42);

\draw[edge, line width=0.3mm] (32) to (43);

\draw[edge, line width=0.3mm] (32) to (44);

\draw[edge, line width=0.3mm] (33) to (45);

\draw[edge, line width=0.3mm] (33) to (46);

\draw[edge, line width=0.3mm] (34) to (47);

\draw[edge, line width=0.3mm] (34) to (48);

\draw[edge, line width=0.3mm] (35) to (49);

\draw[edge, line width=0.3mm] (35) to (410);

\draw[edge, line width=0.3mm] (36) to (411);

\draw[edge, line width=0.3mm] (36) to (412);

\draw[edge, line width=0.3mm] (37) to (413);

\draw[edge, line width=0.3mm] (37) to (414);

\draw[edge, line width=0.3mm] (38) to (415);

\draw[edge, line width=0.3mm] (38) to (416);

\draw[edge, line width=0.3mm] (12) to (11);

\draw[edge, line width=0.3mm] (23) to (11);

\draw[edge, line width=0.3mm] (22) to (21);

\draw[edge, line width=0.3mm] (24) to (23);

\draw[edge, line width=0.3mm] (32) to (31);

\draw[edge, line width=0.3mm] (33) to (21);

\draw[edge, line width=0.3mm] (34) to (33);

\draw[edge, line width=0.3mm] (35) to (11);

\draw[edge, line width=0.3mm] (36) to (35);

\draw[edge, line width=0.3mm] (37) to (23);

\draw[edge, line width=0.3mm] (38) to (37);

\draw[edge, line width=0.3mm] (42) to (41);

\draw[edge, line width=0.3mm] (43) to (31);

\draw[edge, line width=0.3mm] (44) to (43);

\draw[edge, line width=0.3mm] (45) to (21);

\draw[edge, line width=0.3mm] (46) to (45);

\draw[edge, line width=0.3mm] (47) to (33);

\draw[edge, line width=0.3mm] (48) to (47);

\draw[edge, line width=0.3mm] (49) to (11);

\draw[edge, line width=0.3mm] (410) to (49);

\draw[edge, line width=0.3mm] (411) to (35);

\draw[edge, line width=0.3mm] (412) to (411);

\draw[edge, line width=0.3mm] (413) to (23);

\draw[edge, line width=0.3mm] (414) to (413);

\draw[edge, line width=0.3mm] (415) to (37);

\draw[edge, line width=0.3mm] (416) to (415);

\end{tikzpicture}

\caption{\label{fig:uncountable} The modified binary tree.}

\end{figure}
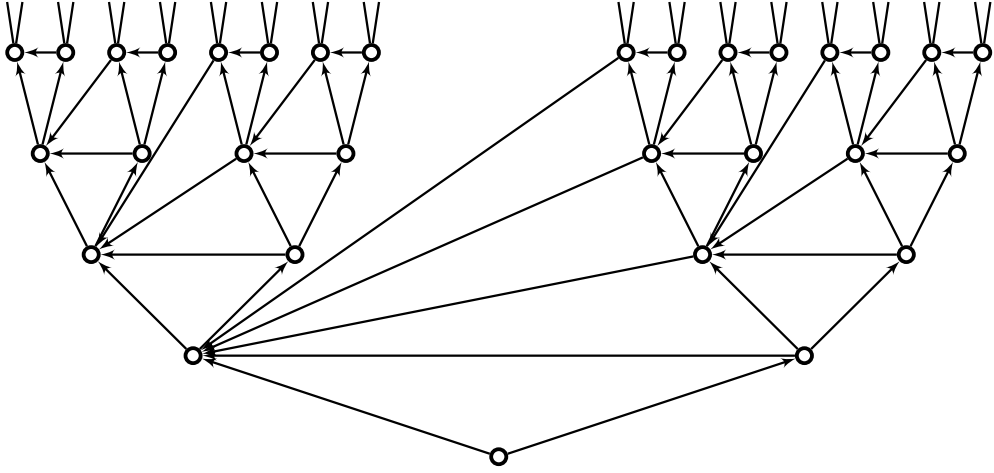
 
We can define an MGWBP and a regular family $\mathcal{A}$ with $G_{\mathcal{A}}=(Z,E_Z)$ in a similar manner as Example 2; however we do not provide an explicit construction here. Note that the graph $(Z,E_Z)$ satisfies the assumptions of Lemma~\ref{lem:canonic}, hence such an MGWBP and regular family $\mathcal{A}$ must exist. 
For simplicity, below we will assume that, as in Example~2, the typeset in our MGWBP is $\mathcal{X}=Z$ and the regular family is $\mathcal{A}=\mathcal{X}$ (the set of singletons).

In this example the set of primitive subsets is $\textswab{P}_{\mathcal{A}} = \mathcal{X}$, i.e., the set of singletons. This is because, by construction, for any $x,y \in \mathcal{A
}$, either $x \Rightarrow y$ or $y \Rightarrow x$.
To identify $\textswab{C}_{\mathcal{A}}$ note that each pure ascending chain corresponds to a ray in the tree, which can be represented by its end point $\{\alpha_i\}_{i=1}^\infty$. The representative of the pure ascending chain is the set of vertices that lie to the right of its corresponding ray.  More formally, for each ray $\{\alpha_i\}_{i=1}^\infty$, we let
$$
h(\{\alpha_i\}_{i=1}^\infty):=\{\emptyset\} \cup \bigcup_{n=1}^\infty \left\{\{\beta_i\}_{i=1}^n: \sum_{i=1}^n 3^{-i} \beta_i\geq \sum_{i=1}^n 3^{-i} \alpha_i\right\}$$
denote the set of vertices to the right of the ray $\{\alpha_i\}_{i=1}^\infty$.
The set of representatives of pure ascending chains is then
\begin{equation}\label{Ex2M1}
\textswab{C}_{\mathcal{A}} =\left\{  h(\{\alpha_i\}_{i=1}^\infty): \{\alpha_i\}_{i=1}^\infty\in \{-1,+1\}^\infty\right\}.\end{equation}
Note that the set $\textswab{C}_{\mathcal{A}}$ is uncountable because there are uncountably many rays.
Here, we have 
\begin{align}
\textswab{J}_{\mathcal{A}} &= \{ (I,J) \in (\textswab{P}_{\mathcal{A}} \times \textswab{C}_{\mathcal{A}}): I \cap J = \emptyset, (J \backslash I^+)^+=J \}  \nonumber \\
&=\{ (I,\emptyset): I \in \textswab{P}_{\mathcal{A}} \}\cup \{ (\emptyset,J): J \in  \textswab{C}_{\mathcal{A}} \}.\label{Ex2M2}
\end{align}
To understand Equation \eqref{Ex2M2}, note that if $I\equiv \{ x\} \in \textswab{P}_\mathcal{A}$ and $J \in \textswab{C}_\mathcal{A}$, then either $x \in J$, in which case $I \cap J \neq \emptyset$, or $x \notin J$, in which case if $y \in J$ then $\{y\} \Rightarrow \{x \}$, thus $J\subseteq I^+$ and therefore $(J \backslash I^+)^+=\emptyset$; thus if $I \cap J = \emptyset$ and  $(J \backslash I^+)^+=J$ then either $I=\emptyset$ or $J=\emptyset$. By Proposition~  \ref{pro:extAeq},  the set of distinct extinction probability vectors is then given by
$$\text{Ext}=\text{Ext}(\mathcal{A})= \{\vc q(I):I \in \textswab{P}_{\mathcal{A}} \cup  \textswab{C}_{\mathcal{A}} \}.
$$
Because $\textswab{C}_{\mathcal{A}}$ is uncountable, by Theorem \ref{th:finiteinfinite2}\emph{(v)}, ${\rm Ext}$ contains uncountably many distinct elements. In addition, because $\mathcal{X}$, and therefore $\textswab{P}_{\mathcal{A}}$, is countable,
we have thus constructed an example in which $|\textswab{P}_{\mathcal{A}}|<|{\rm Ext}(\mathcal{A})|$.
Note that  in this example the inequality in Equation \eqref{bound} is strict.

\section{Open questions}
\label{sec:open}

The results in this paper motivate several open questions. 
Here we consider a very general setting, in which we observe a wide variety of behaviours; for instance, in Example~1, there can be \emph{any} number of distinct extinction probability vectors.
We can then ask whether we observe similarly rich behaviour in more homogeneous settings, such as transitive or quasi-transitive processes. We believe that the answer is negative. In particular, for quasi-transitive BRWs on a graph $G$, like those considered in \cite{Stacey2003} (see also the examples  in \cite{Can15}),
we conjecture that either (i) $|Ext|=1$, in which case $\vc q=\vc{\tilde q}=\bm{1}$, (ii) $|\text{Ext}|= 2$, in which case $\bm{q}=\bm{\tilde q}<\bm{1}$ or $\bm{q}<\bm{\tilde q}=\bm{1}$, or (iii) $|\text{Ext}|$ is uncountable, such as in \cite[Section 3.1]{cf:BZ2020}. Furthermore, we conjecture that, if the process is quasi-transitive, then (iii) can only occur when 
it is nonamenable (see \cite[Section 2.1]{cf:BZ14-SLS} for the definition). 
Note that, without the quasi-transitivity assumption, the MGWBP can exhibit an uncountable number of extinction probability vectors even if both the underlying graph and the process itself are amenable (see Example 1 with $r>1$).
We believe that similar results also hold for irreducible BRWs in an i.i.d. random environment such as those considered in \cite{Com2007, Machado2003}.

Moreover, the exact location of the extinction probability vectors $\vc q(A)$ (different from $\vc q$ and $\vc{\tilde{q}}$) in the set of fixed points $S$ is yet to be identified. In  \cite{Bra18}, the authors conjecture that the ``corners" of the set $S$ correspond to extinction probability vectors $\bm{q}(A)$; see  \cite[Conjecture 5.1]{Bra18} for a precise statement. In addition, it has been shown that $S$ can contain (uncountably many)
fixed points which are not extinction probability vectors; see for instance \cite[Example 3.6]{cf:BZ2015}. Under particular assumptions (i.e. in an irreducible LHBP), it has been shown that there is a continuum of fixed points between $\vc q$ and $\vc{\tilde{q}}$ and there are no fixed points between $\vc{\tilde{q}}$ and $\vc 1$; see \cite[Theorem 1]{Bra19}. Here we prove that there are no fixed points between $\vc{\tilde{q}}$ and $\vc 1$  in the general irreducible setting (Corollary~\ref{cor:0}); we believe that, like in the setting of \cite{Bra19}, there is a continuum of fixed points between $\vc q$ and $\tilde{\vc q}$, however this is yet to be established rigorously. 
Another closely related question is the following: is it possible to have $|\text{Ext}|<|S|<+\infty$? 




Finally, here we focused on the distinct elements of Ext$(\mathcal{A})$, where $\mathcal{A}$ is a regular family. In Example 1, we showed that  Ext$\equiv$Ext$(\mathcal{A})$, and therefore the study of Ext could be reduced to that of Ext$(\mathcal{A})$ without loosing any information. More generally we may ask under which conditions there exists a regular family $\mathcal{A}$ such that Ext$\equiv$Ext$(\mathcal{A})$, and if one exists, can it be described?

\section{Proofs}
\label{sec:proofs}

\begin{proof}[Proof of Proposition~\ref{pro:maximal}]
The usual way to identify the maximal and minimal fixed points of a continuous nondecreasing function in a (partially ordered) set is to
generate iteratively two sequences starting from the maximal and minimal elements of the set (if available).

More precisely, observe that if we let $\bm{\widehat G}^{(A,n+1)}(\bm{s})=\bm{\widehat G}^{(A,n)}(\bm{\widehat G}^{(A)}(\bm{s}))$, then ${\widehat G}_x^{(A,n)}(\bm{1})=\pr_x(\bm{\hat{Z}}_n^{(A)}<\infty)$, that is, ${\widehat G}_x^{(A,n)}(\bm{1})$ is the probability that, given $\bm{\widehat Z}_0=\vc e_x$, no type $y\in A$ individual has been born into the population before generation $n$.
We then have $= {\widehat G}_x^{(A,n)}(\bm{1})\searrow \bm{q}^{(0)}_x(A)$ as $n \to +\infty$. 
The fact that $\bm{q}^{(0)}(A)$ is the unique componentwise \emph{maximal} element of the set $\widehat S^{(A)}$ then follows from the fact that $\bm{\widehat G}^{(A)}(\bm{s})$ (and therefore its iterates) are increasing in $\bm{s}$.

Similarly, ${\widehat G}_x^{(A,n)}(\bm{0})=\pr_x(\bm{\hat{Z}}_n^{(A)}=0)= q_x^{(n)}(\mathcal{X},A)$, and the limit of this nondecreasing sequence (namely $\bm{q}(\mathcal{X},A)$) is necessarily the minimal element of $\widehat{S}^{(A)}$.
\end{proof}

\begin{proof}[Proof of Theorem~\ref{LargeLoc}]

\emph{(i)}. Let us fix $\bm{s}$ such that $\bm{s} \le \bm{G}(\bm{s})$ and suppose 
$s_x<1$ for some $x\in\mathcal{X}$.

Define $\bm{\widetilde G}^{(x)} \colon [0,1]^{\mathcal{X}} \to [0,1]^{\mathcal{X}}$ such that 
\[
\widetilde G^{(x)}_y( \bm{u}) = \begin{cases}
u_x, \quad &y=x, \\
G_y(\bm{u}), \quad &\text{otherwise}.
\end{cases}
\]
Observe that $\bm{\widetilde G}^{(x)}(\cdot)$ is the generating function of the original process modified so that all type-$x$ individuals are \emph{frozen} (at each generation they produce a single copy of themselves).
By induction, for any $n\geq 0$, we have $\bm{s} 
\le
\bm{ \widetilde G}^{(x,n)}( \bm{s})$, which implies $\bm{s} 
\le \lim_{n \to \infty} \bm{ \widetilde G}^{(x,n)}( \bm{s})$. By monotonicity of $\bm{G}(\cdot)$, this leads to $\bm{G}(\bm{s})\le \bm{G}( \lim_{n \to \infty} \bm{ \widetilde G}^{(x,n)}( \bm{s}))$, which implies
\begin{equation}\label{eqGt}
\bm{s} 
\le
\bm{G}(\lim_{n \to \infty} \bm{ \widetilde G}^{(x,n)}( \bm{s})).
\end{equation}
Moreover, the function 
\begin{equation}\label{eqphi}
\phi(s_x)\colon={G}_x(\lim_{n \to \infty} \bm{ \widetilde G}^{(x,n)}(1, \dots, 1, s_x , 1,1,\dots ))
\end{equation}
is the (possibly defective) generating function of the asymptotic number of frozen type-$x$ individuals in the modified process when we start with a single type-$x$ individual in generation 0, and we freeze all type-$x$ individuals after generation $1$. 
If we let this asymptotic number of frozen individuals be $Y_1$ and then repeat these steps, with the initial number of type-$x$ individuals now being $Y_1$, to obtain $Y_2$ and so on, then we obtain a (possibly defective) Galton-Watson process $\{Y_k \}_{k \geq 0}$.
This process is referred to as the \emph{embedded type-$x$ process}, and it is known that
the probability of extinction in $\{Y_k\}$ is ${q_x(\{x\})}$
(see for instance the proof of \cite[Theorem 4.1]{cf:Z1}).
In addition, because $\{\bm{Z}_n \}$ is 
non-singular, $\{Y_k \}$ is non-singular, which means that for any $\varepsilon>0$ and $N<\infty$ there exists $K$ such that 
\begin{equation}\label{qtl}
1 - q_x(\{x\}) - \varepsilon\leq \mbP(Y_k > N) \leq 1 -  q_x(\{x\}) + \varepsilon,
\end{equation}
for all $k \geq K$.
Combining \eqref{eqGt}, \eqref{eqphi} and \eqref{qtl}, we then have $s_x\leq \phi(s_x)$, and for all $k \geq K$,
\begin{align*}
s_x &\le
 \underbrace{\phi \circ \ldots \circ \phi}_k(s_x) = \mbE \big( s_x^{Y_k} \big)  \\
&\leq (s_x)^N (1-q_x(\{x\}) + \varepsilon) + q_x(\{x\}) + \varepsilon.
\end{align*}
For any $\eta >0$ we may then choose $\varepsilon< \eta/2$ and $N$ large enough so that $(s_x)^N < \eta/2$. For these values of $\varepsilon$ and $N$ we can then choose $k$ sufficiently large so that \eqref{qtl} holds. Taking $\eta \downarrow 0$ we then obtain $s_x \leq q_x(\{x\})$. 


\emph{(ii)}. It is not difficult to prove (see for instance the \textit{maximum principle}  \cite[Proposition 2.4]{cf:BZ14-SLS}) that if $s_x<1$ then
$s_y<1$ for all $y \to x$. The previous
part of the theorem yields the claim.

\emph{(iii)}.
In the irreducible case $y \to x$ for all $y \in \mathcal{X}$. Whence $\bm{s} \neq \bm{1}$ implies $s_x<1$ for all $x \in \mathcal{X}$. Again, the first part of the theorem yields the claim because
$\widetilde q_x=q_x(\{x\})$ (see \cite[Corollary 4.1]{Bra18}).
\end{proof}

\begin{proof}[Proof of Corollary~\ref{cor:0}]
 We only prove the equality $\textrm{Ext}=S$ since the rest follows trivially from
 Theorem~\ref{LargeLoc}. 
If $\bm{s}=\bm{1}$ then there is nothing to prove; otherwise consider the (non-empty) set $A:=\{x \in \mathcal{X}\colon s_x<1\}$; we prove that $\bm{s}=\bm{q}(A)$ (which shows that any fixed point is an extinction probability vector). First,  by definition of $A$ and by the \textit{maximum principle} \cite[Proposition 2.4]{cf:BZ14-SLS}, there are no $y \in A^c$ and $x \in A$ such that $y \to x$. Therefore $q_y(A)=1=s_y$ for all $y \in A^c$. On the other hand, if $x\in A$ then $\bm{q}(A) \leq \bm{q}(\{x\})$; moreover
 by Theorem~\ref{LargeLoc}\emph{(i)}, $q_x\leq s_x\leq q_x(\{x\})$, and we also have
  $q_x \le q_x(A)
 \le q_x(\{x\})$ for all $x \in A$, which yields the conclusion.
\end{proof}

\begin{proof}[Proof of Theorem~\ref{th:1}]
We start by proving the equivalence \emph{(i)} $\Leftrightarrow$ \emph{(iii)}.
Theorem 2.4 in \cite{cf:BZ2020} implies that,
for every fixed point $\vc s$,
$q_x(B)>s_x$ for some $x\in \mathcal{X}$ 
if and only if $q_y^{(0)}(B)>s_y$
for some $y \in \mathcal{X}$.
 It is enough to take
$\bm{s}=\bm{q}(A)$.

The implication {\emph{(iii)} $\Rightarrow$ \emph{(iv)}} is trivial, since the probability of survival in $A$ is strictly larger than the probability of visiting $B$. The implications
{\emph{(iv)} $\Rightarrow$ \emph{(v)}} and {\emph{(vi)} $\Rightarrow$ \emph{(i)}}
are also straightforward.


We now prove that {\emph{(v)} $\Rightarrow$ \emph{(vi)}}. Suppose $\mbP_x(\mathcal{S}(A) \cap \mathcal{E}(B))>0$ and fix $x$ as the type of the initial individual.
Let $\mathcal{F}_n$ denote the history of the process up to generation $n$ and observe that 
\begin{align*}
M_n(A) &:= \mbP_x(\mathcal{E}(A) | \mathcal{F}_n )=\mbP_x(\mathcal{E}(A) | \vc Z_n ) = \bm{q}(A)^{\bm{Z}_n} \\
M_n(B) &:= \mbP_x(\mathcal{E}(B) | \mathcal{F}_n )=\mbP_x(\mathcal{E}(B) | \vc Z_n )  = \bm{q}(B)^{\bm{Z}_n} 
\end{align*}
are martingales.
By Doob's martingale convergence theorem $M_n(A) \to \mbP_x(\mathcal{E}(A) | \mathcal{F}_\infty ) = \mathds{1}_{\mathcal{E}(A)}$ as $n \to \infty$, with the same holding for extinction in $B$.
Thus, by assumption
\begin{equation}\label{Eqn1}
\mbP_x(\mathcal{S}(A) \cap \mathcal{E}(B))=\mbP_x\bigg(\lim_{n \to \infty} \bm{q}(A)^{\bm{Z}_n}=0, \,\lim_{n \to \infty} \bm{q}(B)^{\bm{Z}_n}=1\bigg)>0. 
\end{equation}
Now, suppose by contradiction that there exists $c>0$ such that 
\begin{equation}\label{Eqn2}
{1-q_i(B)}\geq c({1-q_i(A)}) 
\end{equation}
uniformly in $i \in \mathcal{X}$.
Then,
\begin{align}
\bm{q}(B)^{\bm{Z}_n} &= \prod_{i \in \mathcal{X}} (1-(1-q_i(B)))^{Z_{n,i}} \nonumber \\
&\leq  \prod_{i \in \mathcal{X}} (1-c(1-q_i(A)))^{Z_{n,i}} \nonumber \\
&\leq \exp\bigg\{ -c \sum_{i \in \mathcal{X}} Z_{n,i}(1-q_i(A)) \bigg\}, \label{Eqn3}
\end{align}
where to obtain \eqref{Eqn3} we use the fact that $1-y \leq e^{-y}$.
In addition, 
using the inequality
$1-\prod_{i \in I} \alpha_i \le \sum_{i \in I} (1-\alpha_i)$ (where $I$ is 
countable and $\alpha_i \in [0,1]$ for all $i \in I$) and the subadditivity of the probability measure,
%
%
%
%
%
\[
1- \bm{q}(A)^{\bm{Z}_n} \leq  \sum_{i \in \mathcal{X}}  (1-q_i(A)^{Z_{n,i}})\leq \sum_{i \in \mathcal{X}} Z_{n,i} (1-q_i(A))
\]
so that
\begin{equation}\label{Eqn4}
\bm{q}(A)^{\bm{Z}_n} \geq 1- \sum_{i \in \mathcal{X}} Z_{n,i} (1-q_i(A)).
\end{equation}
Combining \eqref{Eqn3} and \eqref{Eqn4} we obtain 
\begin{align*}
\mbP_x\bigg(\lim_{n \to \infty} &\bm{q}(A)^{\bm{Z}_n}=0, \,\lim_{n \to \infty} \bm{q}(B)^{\bm{Z}_n}=1\bigg) \\
&\leq \mbP_x \bigg( \liminf_{n \to \infty} \sum_{i \in \mathcal{X}} Z_{n,i}(1-q_i(A)) \geq 1, \, \lim_{n \to \infty} \sum_{i \in \mathcal{X}} Z_{n,i}(1-q_i(A))=0 \bigg) \\
&=0,
\end{align*}
which contradicts \eqref{Eqn1}. Thus, the assertion in \eqref{Eqn2} cannot hold.


The equivalence {\emph{(i)} $\Leftrightarrow$ \emph{(ii)}} follows from the equality
 $\mathcal{S}(A)\cap \mathcal{E}(B)=
\mathcal{S}(A\setminus B)\cap \mathcal{E}(B)$ and the fact that {\emph{(v)} $\Leftrightarrow$ \emph{(i)}} (apply {\emph{(v)}} with $A\setminus B$ instead of $A$).

Finally, we prove {\emph{(iv)} $\Leftrightarrow$ \emph{(vii)}}. Assume $A=\mathcal{X}$.
Since $\hat S^{(B)}$ is non-empty, by Proposition~\ref{pro:maximal} it is not a singleton if and only
if $\bm{q}(\mathcal X,B)<\bm{q}^{(0)}(B)$. Note that
$q_x^{(0)}(B)-q_x(\mathcal X,B)=\pr_x(\mathcal{N}(B) \cap \mathcal{S}(\mathcal X))$ whence
$\bm{q}(\mathcal X,B)<\bm{q}^{(0)}(B)$ if and only if there exists $x \in \mathcal{X}$ such that 
$\pr_x(\mathcal{N}(B) \cap \mathcal{S}(\mathcal X))>0$,
that is, if and only if \emph{(iv) holds}.
\end{proof}

\begin{proof}[Proof of Corollary~\ref{cor:1}]
If $\sup_{x \in \mathcal{X}} {q}_x = 1$ there is nothing to prove. Otherwise, suppose $\sup_{x \in \mathcal{X}} {q}_x < 1$; then by Theorem \ref{th:1} \textit{(vi)} (set $A=\mathcal{X}$ and $B=A$), we have
\[
 \inf_{x \in \mathcal{X}} (1-q_x(A))
 \le
 \inf_{x \in \mathcal{X}} \frac{1-q_x(A)}{1-q_x}=0,
\]
which yields the claim.
\end{proof}

\begin{proof}[Proof of Theorem~\ref{th:equivalence}]
The equivalence $(ii) \Leftrightarrow (iii)$ follows from
Corollary~\ref{cor:equivalence}.

Suppose that (C4) holds.
Let us prove that $(i) \Rightarrow (ii)$. Since for all $i \in I$ we have $A_i \Rightarrow A_j$ for some $j \in J$, then
$A_i \Rightarrow  \bigcup_{j \in J} A_j$ for all $i \in I$ which, by (C4), implies
$\bigcup_{i \in I} A_i \Rightarrow 
\bigcup_{j \in J} A_j$. By exchanging the role of $I$ and $J$ we prove the claim.
This implies that the map $[I]_\sim \mapsto \bm{q}\big ( \bigcup_{i \in I} A_i \big )$ is well defined and, if (C1) holds, it is a surjective map onto $\text{Ext}(\mathcal{A})$.

Now assume (C2) and (C5). We prove that $(ii) \Rightarrow (i)$.
Suppose that either $I$ or $J$ are empty; then \textit{(i)} holds if and only if they are both empty. The same holds for \textit{(iii)} and \textit{(ii)} because $\bm{q}\big ( \bigcup_{i \in I} A_i \big )=\bm{1}$ if and only if $I=\emptyset$.
We can assume henceforth $I, J \neq \emptyset$.
Suppose, by contradiction, that there exists $i \in I$ such that
$A_i \nRightarrow A_j$ for all $j \in J$
(if there exists $j \in J$ such that
$A_j \nRightarrow A_i$ for all $i \in I$ we proceed analogously): in this case $J \subseteq J_{i}$ and, by (C5),
$A_i \nRightarrow \bigcup_{j \in J} A_j$. This implies $\bigcup_{i \in I} A_i \nRightarrow 
\bigcup_{j \in J} A_j$ and yields the claim.
Moreover, it implies that $\bm{q}\big ( \bigcup_{i \in I} A_i \big ) \mapsto [I]_\sim$ is a well defined surjective map from a subset of $\text{Ext}(\mathcal{A})$ onto
$2^{K_{\mathcal{A}}}/_\sim$.
\end{proof}

\begin{proof}[Proof of Lemma~\ref{lem:canonic}]
 Fix a family of probability distributions
 $\{r_i\}_{i \in Z}$, where $r_i=\{r_{ij}\}_{j \in Z}$ such that $r_{ii}=1/2$ for all $i \in Z$ and, when $i \neq j$, $r_{ij}>0$ if and only if $(i,j) \in E_Z$. Consider a probability generating function $\phi(s)$ such that $\phi^\prime(1)>2$. 
 
We define a MGWBP on $Z$ by the following reproduction rules: a particle living at $i$ produces a random number of offspring according to the distribution with 
probability generating function $\phi$; each newborn particle is placed at random independently according to the distribution $r_{i}$. The offspring generating function of this MGWBP is $G_i(\vc s):=\phi(\sum_j r_{ij} s_j)$.
 Define the family $\mathcal{A}$ as the collection
  of singletons $A_i:=\{i\}$ for $i\ \in Z$. 
  
  Clearly local survival in $i$ implies survival in $j$ if and only if there is a path from $i$ to $j$ in $(Z,E_Z)$.
   Let us prove regularity. Condition $(C1)$ is trivial and, since there are no closed
paths in $(Z,E_Z)$, then Condition $(C3)$ follows.

The probability of local extinction starting from $i$ is the smallest nonnegative fixed point of the generating function $\psi(s):=G_i(\vc s)|_{s_i=s, s_j=1, j\neq i}$; indeed, every child placed outside $i$ cannot contribute to the local survival (because there are no closed paths of length strictly larger than 1). This means that each particle in the progeny has the same (positive) probability $1-\beta$ of generating a population which survives locally and this implies $(C2)$. 
  

Let us pick $I \subseteq Z$. If the process survives in
$\bigcup_{i \in I} A_i$ then there are infinitely many descendants, and by a Borel-Cantelli argument, almost surely, at least one of them (actually an infinite number of them) will generate a progeny which survives locally.
Thus, for every fixed $I$, survival in $\bigcup_{i \in I} A_i$ implies survival in $A_i$ for some $i \in I$. 
This proves that Condition (C4) holds.

To prove $(C5)$ it is enough to observe  that
$A_i \nRightarrow A_j$ if and only if there is no path from $i$ to $j$  in $(Z,E_Z)$;
thus, if the process starts from $i$, then the probability of visiting $\bigcup_{j \in J_i} A_j=J_i$ is 0, while
the probability of survival in $A_i$ is strictly positive.
\end{proof}

\begin{proof}[Proof of Lemma~\ref{lem:IM}]
Recall that, by definition, $[I_d]_\sim=[I_M]_\sim$, that is, $I_d \sim I_M$.

\emph{(i)}.
If $I_c=\emptyset$ then 
$I=I_d$ and $[I]_\sim=[I_d]_\sim=[I_M]_\sim$.
Conversely, if $I \sim I_M$ then for all $i \in I$ there exists $j \in I_M$ such that $A_i \Rightarrow A_j$, thus $i \in I_d$. This implies that $I_C =\emptyset$.

\emph{(ii)}.
The claim follows from the chain of equalities
$[I_d]_\sim=[I_M]_\sim=[J_M]_\sim=[J_d]_\sim$. 

\emph{(iii)}.
If $I_M=J_M$ then $I_d \sim J_d$ by \textit{(ii)}. 
Conversely, since $[I_M]_\sim=[I_d]_\sim=[J_d]_\sim=[J_M]_\sim$,  $I_M$ and $J_M$ are primitive subsets, and (C3) holds, we have $[I_M]_\sim=[J_M]_\sim$, which implies $I_M=J_M$ because these sets are primitive.

%
%
%
\emph{(iv)}.
Let us prove
$\Longrightarrow$.
Let $I \sim J$ and $i \in I_M$. If $j \in J$ such that $A_i \Rightarrow A_j$, there exists $i_1 \in I$ such that $A_j \Rightarrow A_{i_1}$, thus $A_i \Rightarrow A_{i_1}$ whence
$i=i_1=j$ (from the definition of $I_M$ and from (C3)). 
Since by the equivalence there exists such a $j \in J$, we have that $i$ is an element of $J$ which does not imply any other element of $J$, that is, 
$i \in J_M$. Thus $I_M \subseteq J_M$; by exchanging the role of $I$ and $J$, we have $I_M=J_M$.
 For all $i \in I_c$, there exists $j \in J$ such that $A_i \Rightarrow A_j$ and, by the definition of $I_c$, there is no $l \in I_M$ such that $A_j  \Rightarrow A_l$. Since $I_M=J_M$ then $j \in J_c$. By exchanging the role of $I$ and $J$ we have 
$I_c \sim J_c$. \hfill

Let us now prove
$\Longleftarrow$.
Let $i \in I$. If $i \in I_M=J_M$ then $i \in J$. If $i \in I_d$ then, since $I_d \sim J_d$, $A_i \Rightarrow A_j$ for some $j \in J_d \subseteq J$, whence $A_i \Rightarrow A_j$ for some $j \in J$. By exchanging the role of $I$ and $J$, we have that 
for all $j \in J$ there exists $i \in I$ such that $A_j \Rightarrow A_i$. This proves that $I \sim J$.


\emph{(v)}.
Note that, from \textit{(iv)}, if $I \sim J$ then $I_c=\emptyset$ if and only if $J_d=\emptyset$. Whence, if $J$ is primitive and $I \sim J$ we have $\emptyset=J_c=I_c$. The converse follows from \textit{(i)} by taking $I:=J_M$.

\emph{(vi)}.
We prove, by induction, that there is a sequence of pairwise distinct elements
$\{i_n\}_{n \in \N}$ such that, for all $n \in \N$,  $i_n \in I_c$ and $A_{i_n} \Rightarrow A_{i_{n+1}}$. Since $I_c \neq \emptyset$ there exists $i_0 \in I_c$. Suppose that we have $n+1$ distinct elements $i_0, i_1, \ldots, i_{n} \in I_c$ satisfying the above relation. Since $i_n \not \in I_M$ there exists $i_{n+1} \neq i_{n}$ in $I$ such that $A_{i_{n}} \Rightarrow A_{i_{n+1}}$. By (C3), $i_{n+1} \neq i_j$ for all $j <n$.  If $j \in I$ such that $A_{i_{n+1}} \Rightarrow A_j$ then $A_{i_{n}} \Rightarrow A_j$ whence $j \neq I_M$ since $i_n \in I_c$; this implies that $i_{n+1} \in 
I_c$.

\emph{(vii)}.
It follows easily from \textit{(iii)} and \textit{(iv)}, from the decomposition $I=I_d \cup I_c$ and 
from the basic properties discussed above, $(J_d)_M=(J_M)_M=J_M$, $(J_d)_d=J_d$, $(J_M)_c=(J_d)_c=(J_c)_M=(J_c)_d=\emptyset$ which hold for all $J \subseteq K_\mathcal{A}$.
\hfill
\break
\noindent
Suppose that $J \sim I$ and consider the decomposition $J=J_d \cup J_c$. Observe that, from \textit{(iv)} and the basic properties discussed above, $(J_d)_M=J_M=I_M$, $(J_d)_c=(J_c)_M=(J_c)_d=\emptyset$. By taking $H:=J_d$ and $W:=J_c$ we prove that $J$ belong to the set in the right-hand-side.
\hfill
\break
\noindent
Conversely, let $J=H \cup W$ belong to the set in the right-hand-side, and let us prove that $J\sim I$. If $i \in I_d$ then $A_i \Rightarrow A_j$ for some $j \in I_M=H_M \subseteq J$ then $A_i \Rightarrow A_j$ where $j \in J$. If $i \in I_c$ then, by hypotheses there exists $j \in W_c \subseteq J$ such that $A_i \Rightarrow A_j$. If $j \in J_d$
\end{proof}

\begin{proof}[Proof of Proposition~\ref{pro:cardinality}]
 Assume (C3).   
  We make use of Lemma~\ref{lem:IM} to show that the map
$g_\mathcal{A}([I]_\sim)=(I_M, (I_c)^+)$ is a bijection from
$2^{K_\mathcal{A}}/_\sim$ onto 
$\textswab{I}_\mathcal{A}$.
   %
   %
The map is well defined and injective by Lemma~\ref{lem:IM}\emph{(iv)}; indeed,
note that $I_c \sim \tilde I_c$ if and only if $(I_c)^+=(\tilde I_c)^+$. By the definition of $I_c$ it is clear
that for all  $i\in I_M$ and $j \in I_c$ we have $A_i \not \Leftrightarrow A_j$, whence the image of the map is a subset of $\textswab{I}_\mathcal{A}$ (take $J=(I_c)^+$ and $\bar J=I_c$ in Equation~\eqref{eq:Ia}).

We are left to prove that the map is surjective.
Note that $\textswab{I}_\mathcal{A}$ can be equivalently
defined as
\begin{equation}\label{eq:Ia2}
\textswab{I}_\mathcal{A}:=\{(I, J) \in \left(\textswab{P}_\mathcal{A} \times \textswab{C}_\mathcal{A}\right)
\colon  
 \exists \bar J \sim J, 
  \forall i \in I,\, \forall j \in \bar J,\, A_i \not \Leftrightarrow A_j 
\}.
\end{equation}
Let $(I, J) \in \textswab{I}_\mathcal{A}$ and let $\bar J \sim J$ such that $\forall i \in I, \forall j \in \bar J, A_i \not \Leftrightarrow A_j$. If we define $\bar I:=I \cup \bar J$ we have that $\bar I_M=\bar I_d= I$ and $\bar I_c=\bar I \setminus \bar I_d=\bar J$; clearly $(\bar J)^+=J$ since $\bar J \sim J$. Then $g_\mathcal{A}([\bar I]_\sim)=(I, (\bar J)^+)=(I,J)$, whence the map is surjective and \textit{(i)} is proved.

When $\textswab{C}_\mathcal{A}=\{\emptyset\}$ then $\textswab{I}_\mathcal{A}=\textswab{P}_\mathcal{A} \times \{\emptyset\}$. The claim \textit{(ii)} follows by the equality
$f_\mathcal{A}=g^{-1}_\mathcal{A}\circ h$, where $h$ is the natural bijection from $\textswab{P}_\mathcal{A}$ onto  $\textswab{P}_\mathcal{A} \times \{\emptyset\}$.

\end{proof}

%

The proofs of Theorem~\ref{th:finiteinfinite2} and Corollary~\ref{th:2} require the following lemma.

 \begin{lemma}\label{pro:finiteinfinite1}
Let $\mathcal{A}=\{A_1, A_2, \ldots, A_{\kappa_{\mathcal{A}}}\}$ with $\kappa_\mathcal{A} \leq +\infty$.
 \begin{enumerate}[(i)]
   \item $\kappa_{\mathcal{A}}<\infty$ if and only if $\textswab{P}_\mathcal{A}$ is finite. 
   \item If $\kappa_{\mathcal{A}}=\infty$, then the following statements are equivalent:
\begin{enumerate}[(1)]
\item $\textswab{P}_\mathcal{A}$ is uncountable;
\item there exists an infinite, primitive $I$;
\item there exists a family $\mathcal{A}'\subseteq \mathcal{A}$ such that 
$A'_i \nLeftrightarrow A'_j$ for all $A'_i, A'_j \in \mathcal{A}'$;
\end{enumerate}
    \end{enumerate}
\end{lemma}

\begin{proof}[Proof of Lemma~\ref{pro:finiteinfinite1}]

\emph{(i)}.  Clearly if $\kappa_\mathcal{A}<+\infty$ is finite then $\textswab{P}_\mathcal{A}$ is finite as well.
Conversely, 
 since  every singleton $\{i\}$, where $i \in K_{\mathcal{A}}$, is a primitive set, the reverse implication holds.

\emph{(ii)}. If $I \subseteq K_{\mathcal{A}}$ is primitive and infinite, then it must be countably infinite; in this case, every subset of $I$ is primitive, and the collection of all subsets of $I$ is uncountable, thus we have $\emph{(2)}\Rightarrow\emph{(1)}$. To prove $\emph{(1)}\Rightarrow\emph{(2)}$, it is enough to note that  $\{I\subseteq K_{ \mathcal{A}} \colon I \text{ primitive and finite} \} \subseteq \bigcup_{i=0}^\infty K_{\mathcal{A}}^i$ which is a countably infinite set (provided $\kappa_{\mathcal{A}} \neq 0$). The implication $\emph{(2)}\Rightarrow \emph{(3)}$ is straightforward if we take $\mathcal{A}':=\{A_i \colon i \in I\}$. To prove $\emph{(3)}\Rightarrow \emph{(2)}$, just take $I:=\{i \in K_\mathcal{A} \colon A_i \in \mathcal{A}'\}$.
\end{proof}

\begin{proof}[Proof of Theorem~\ref{th:finiteinfinite2}]

\emph{(i).}  Clearly if $\mathcal{A}$ is finite then $\Sigma(\mathcal{A})$ is finite and $\text{Ext}(\mathcal{A})$ is finite as well (there is no need for (C3) to hold here). 
Conversely, if $\text{Ext}(
 \mathcal{A})$ is finite, then
by Theorem~\ref{th:equivalence}
(which holds without any assumptions on singletons), we have that
$2^{K_\mathcal{A}}/_\sim$ is finite. By Lemma~\ref{lem:primitive}, $\textswab{P}_\mathcal{A}$ is finite as well. By Lemma~\ref{pro:finiteinfinite1}\emph{(i)} $\kappa_\mathcal{A}$ is finite.
  
\emph{(ii).}  We observed that, if (C2), (C3) and (C5) hold, then by Lemma~\ref{lem:primitive} and  Theorem~\ref{th:equivalence}, there is an injective map from $\textswab{P}_\mathcal{A}$ into the set $\text{Ext}(\mathcal{A})$, and this yields Equation~\eqref{bound}.
By regularity, according to Theorem~\ref{th:equivalence}, 
$|\text{Ext}(\mathcal{A})|=
|2^{K_\mathcal{A}}/_\sim|$.
If, in addition there are no ascending chains, by Proposition~\ref{pro:cardinality}, we have equality in Equation~\eqref{bound}.

\emph{(iii)}. Suppose Ext$(\mathcal{A})$ is countably infinite which, as shown above, implies $\kappa_{\mathcal{A}}=\infty$.  By Lemma~\ref{pro:finiteinfinite1}\emph{(ii)} and Equation~\eqref{bound},  an infinite primitive $I \subseteq K_{\mathcal{A}}$ does not exist.
Consider the graph $G_{\mathcal{A}}$ on $K_{\mathcal{A}}$; let $I_0=\emptyset$ and for $j\geq 1$, define $I_j$ recursively so that $I_j$ is the set of vertices with out-degree zero in the induced graph $G_{\mathcal{A}}[ K_{\mathcal{A}} \backslash (\bigcup_{i=0}^{j-1} I_i)]$.
By construction, $E_{\mathcal{A}}[I_j] = \emptyset$ for all $j$, thus there cannot exist $j$ such that $|I_j|=\infty$, since $I_j$ is primitive.
In that case, either there exists $j_0\geq 1$ such that $|I_j|=0$ for all $j \geq j_0$, or $0<|I_j|<\infty$ for all $j \geq 1$. In the former case, because $\kappa_{\mathcal{A}}=\infty$ and the graph contains no cycles in $G_{\mathcal{A}}$, there must exist an infinite path $A_{i_1} \Rightarrow A_{i_2} \Rightarrow A_{i_3} \Rightarrow \dots$. In the latter case, since for all $l \in I_{n+1}$ there exists $r \in I_n$ such that $A_l \Rightarrow A_r$, by transitivity we have that
for all $l \in \bigcup_{j >i} I_j$ there exists $r \in I_i$ such that $A_l \Rightarrow A_r$.  Since $0<|I_n|<+\infty$ for all $n$ and the sets $\{I_n\}_n$ are pairwise disjoint, we have that $\big | \bigcup_{j >n} I_j \big |=+\infty$ for all $n$. Besides,
we have that for all $n$,
there exists $i \in I_n$ such that 
$d_i:=|\{r \in \bigcup_{j >n} I_j \colon A_r \Rightarrow A_i \}|$ is infinite. Clearly,
given any $i_n \in I_n$ such that $d_{i_n}=+\infty$, there exists $i_{n+1} \in I_{n+1}$ such that $d_{i_{n+1}}=+\infty$ and $A_{i_{n+1}} \Rightarrow A_{i_n}$.
It is possible to construct iteratively a sequence $\{i_n\}_n$ such that $i_n \in I_n$, $d_{i_n}=+\infty$, and $A_{i_{n+1}} \Rightarrow A_{i_n}$.
In both cases the family $\mathcal{A}'=\{A'_1,A'_2,\ldots \}$, where  $A'_n := A_{i_n}$, satisfies (C2), (C3) and (C5). Moreover it is regular if $\mathcal{A}$ is regular.

\emph{(iv) and (v)}. According to Proposition~\ref{pro:cardinality} and  Theorem~\ref{th:equivalence}, regularity implies that $\text{Ext}(\mathcal{A})$, $2^{\mathcal{A}}/_\sim$, and $\textswab{I}_\mathcal{A}$ have the same number of (distinct) elements. By Equation~\eqref{eq:Ia} and the remarks thereafter,
\[
 \max(|\textswab{P}_\mathcal{A}|,|\textswab{C}_\mathcal{A}|) \le 
 |\textswab{I}_\mathcal{A}| \le
 |\textswab{P}_\mathcal{A} \times \textswab{C}_\mathcal{A}|.
\]
By recalling that $|\textswab{P}_\mathcal{A}| \ge 1$ and $|\textswab{C}_\mathcal{A}| \ge 1$, it is easy to show that $\max(|\textswab{P}_\mathcal{A}|,|\textswab{C}_\mathcal{A}|)$ and $|\textswab{P}_\mathcal{A} \times \textswab{C}_\mathcal{A}|$ are simultaneously finite, simultaneously countably infinite, or simultaneously uncountable. Whence, 
the above double inequality yields the following
table, 
\begin{center}
\begin{tabular}{l|lll}
$\ $ & $|\textswab{P}_\mathcal{A}|<\infty$ &
$\textswab{P}_\mathcal{A}$ countably $\infty$ &
$\textswab{P}_\mathcal{A}$ uncountable \\[0.2em]
\hline
$|\textswab{C}_\mathcal{A}|<\infty$ &
$ |\textswab{I}_\mathcal{A}|<\infty$  &
$ \textswab{I}_\mathcal{A}$ countably $\infty$ &
$ \textswab{I}_\mathcal{A}$ uncountable \\[0.2em]
$\textswab{C}_\mathcal{A}$ countably $\infty$ &
$ \textswab{I}_\mathcal{A}$ countably $\infty$ &
$ \textswab{I}_\mathcal{A}$ countably $\infty$ &
$ \textswab{I}_\mathcal{A}$ uncountable \\[0.2em]
$\textswab{C}_\mathcal{A}$ uncountable &
$ \textswab{I}_\mathcal{A}$ uncountable &
$ \textswab{I}_\mathcal{A}$ uncountable &
$ \textswab{I}_\mathcal{A}$ uncountable, \\
\end{tabular}
\end{center}
and this proves the claims.

\end{proof}

\begin{proof}[Proof of Corollary~\ref{th:2}]
Let $\mathcal{A}=\{A_1,A_2,A_3,\ldots\}$. By assumption, for all $i\geq 1$, $A_i\nRightarrow \mathcal{X} \backslash A_i$, so $A_i\nRightarrow A_j$ for $j\neq i$. This implies that $\mathcal{A}$ is regular and $I=\mathbb{N}$ is a primitive set. Therefore, by Lemma~ \ref{pro:finiteinfinite1}\emph{(ii)} and Theorem \ref{th:finiteinfinite2}\emph{(i)}, $\text{\text{Ext}}(\mathcal{A}) $ is uncountable.
\end{proof}

\begin{proof}[Proof of Proposition~\ref{pro:ultimate}]
\emph{Assume $r\leq1$ (cases {(i)} and {(ii)}).}
We consider the family $\mathcal{A}=\{\mathcal{L}_0,\mathcal{L}_1,\mathcal{L}_2,\ldots\}$ and start by showing that for any $i\geq 1$, $\mathcal{L}_{i-1} \Leftarrow \mathcal{L}_{i}$, that is, survival in $\mathcal{L}_{i}$ implies survival in $\mathcal{L}_{i-1}$, regardless the initial type. This implies $\bm{q}(\mathcal{L}_{i-1}) \leq \bm{q}(\mathcal{L}_{i})$.
%
%
%
%

Observe that, with probability one, an initial $(i,j)$-type individual has an infinite line of descent made of all $(i,k)$-types for $k\geq j$. Let $\xi_k$ denote the geometric number of $(i-1,k)$-type offspring born to the $(i,k)$-type individual in this line of descent. We have $$\sum_{k\geq j} P[\xi_k\geq 1]=\sum_{k\geq j} \frac{1}{1+r^{k-1}}=\infty;$$
because this sum is infinite for all $j\geq 0$, by the Borel-Cantelli Lemma, if the process ever reaches $\mathcal{L}_i$, then with probability 1,  there are infinitely many individuals in $\mathcal{L}_i$ who have at least one child in $\mathcal{L}_{i-1}$; thus, survival in $\mathcal{L}_{i}$ implies survival in $\mathcal{L}_{i-1}$.

We note that global survival implies survival in $\bigcup_{0\leq i\leq \ell}\mathcal{L}_i$ for some $\ell\geq 1$; in particular, global survival implies survival in $\mathcal{L}_1$, and therefore in $\mathcal{L}_0$. This leads to $\vc q=\vc q(\mathcal{L}_0)=\vc q(\mathcal{L}_1)$.

Next, we show that the study of Ext can be reduced to the study of Ext$(\mathcal{A})$: in other words, for any subset $A\subseteq \mathcal{X}$, if  $\iota(A)=\infty$ then $\vc q(A)=\tilde{\vc q}$, while if $\iota(A)<\infty$ then $\bm{q}(A)= \bm{q}(\mathcal{L}_{\iota(A)})$.

We first assume that $\iota(A)=\infty$. If $|A|<\infty$, then clearly $\bm{q}(A)= \bm{\tilde q}$ since the process is irreducible, so we take $|A|= \infty$. In this case survival in $A$ implies survival in $\mathcal{P}_0$. To see why, suppose there is a positive chance of survival in $A$. If, by contradiction, the process became extinct in $\mathcal{P}_0$ there would exist a finite maximum level $K$ ever reached by the process. Since $\iota(A) = \infty$, we would have $|A \cap (\bigcup_{i=0}^K \mathcal{L}_i)|< \infty$, thus survival in $A$ and extinction in $\mathcal{P}_0$ would imply that the process survives locally. However, by irreducibility, local survival implies survival in $\mathcal{P}_0$ which yields a contradiction. Hence $\bm{q}(A) \geq \bm{q}(\mathcal{P}_0)$.
To show $\bm{\tilde q} \equiv \bm{q}((0,0)) = \bm{q}(\mathcal{P}_0)$ first observe that, by Theorem~\ref{LargeLoc}(iii), $\bm{\tilde q}\ge \bm{q}(\mathcal{P}_0)$.
On the other hand, extinction in $(0,0)$ implies that a finite number of particles will ever reach $(0,0)$, and since each of them reaches a finite level in $\mathcal{P}_0$ almost surely, there is almost sure extinction in $\mathcal{P}_0$. When $\iota(A) = \infty$ we therefore have $ \bm{\tilde q} \geq \bm{q}(A) \geq \bm{q}(\mathcal{P}_0) =  \bm{\tilde q}$.

We now assume $1\leq \iota(A)<\infty$. First, observe that survival in $\mathcal{L}_i$ implies survival
in $A$ whenever $|\mathcal{L}_{i} \cap A|=\infty$; for instance when $i=\iota(A)$. Next, we show that survival in $A$ implies survival in $\mathcal{L}_{\iota(A)}$; by definition of $\mathcal{L}_{\iota(A)}$, $A$ only contains a finite number of types in the levels below $\mathcal{L}_{\iota(A)}$, namely the types in $A_1:=A\cap \cup_{i < \iota(A)}\mathcal{L}_{i}$. Therefore, survival in $A$ implies survival in at least one
of $A_1$ and $A\backslash A_1$. 
By the argument above, survival in $A_1$ implies local survival, which implies survival in $\mathcal{L}_{\iota(A)}$. 
On the other hand, survival in $A\backslash A_1$ also implies survival in $\mathcal{L}_{\iota(A)}$ because survival in $\mathcal{L}_{\ell}$ implies survival in $\mathcal{L}_{\ell-1}$ for all $\ell\geq 1$ when $r\leq 1$. So $\bm{q}(A)= \bm{q}(\mathcal{L}_{\iota(A)})$.

Finally, if $\iota(A)=0$, then extinction in $A$ implies extinction in $\mathcal{L}_1$, and therefore extinction in $\mathcal{L}_0$; indeed, by the above argument, survival in $\mathcal{L}_1$ implies that infinitely many individuals in $\mathcal{L}_1$ will have at least one child in $A\cap \mathcal{L}_0$. On the other hand, here $A_1=\emptyset$, and survival in $A$ implies survival in $\mathcal{L}_0$ for the same reason as above. So $\bm{q}(A)= \bm{q}(\mathcal{L}_{0})$.

Thus, for $r \leq 1$, we have at most a countable number of distinct extinction probability vectors.

\emph{Assume $r=1$ (case {(ii)}).}
We show that when $r=1$, the family $\mathcal{A}\setminus\{\mathcal{L}_0\}$ is regular, and due to the linear structure of $G_\mathcal{A}$, the edgeless subgraphs are precisely the countably infinite singletons (individual levels).
It is enough to prove that for any $i \geq 1$, $\mathcal{L}_i\nRightarrow \mathcal{L}_{i+1}$, that is, there exists $x \in \mathcal{X}$ such that $q_x(\mathcal{L}_i) < q_x(\mathcal{L}_{i+1})$. This implies $\bm{q}(\mathcal{L}_i) < \bm{q}(\mathcal{L}_{i+1})$.
It suffices to show that, starting from $x$, there is a positive chance of survival in $\mathcal{L}_i$ without ever reaching $\mathcal{L}_{i+1}$. We consider a $(i,k+1)$-type individual ($k\geq 0$) and note that the expected number of its descendants that eventually reach t $(0,k+j)$ when all particles are frozen as soon as they reach $\mathcal{L}_0$, is $\binom{i+j-1}{j-1}$. 
Each frozen particle at $(0,k+j)$ independently has probability $p^{k+j}$ of having a descendant that reaches $(0,0)$; we refreeze the particles reaching $(0,0)$.  Thus, the expected number of frozen $(0,0)$-type descendants of the initial $(i,k+1)$-type individual is given by 
\[
p^{k}\sum_{j =1}^\infty \binom{i+j-1}{ j-1} p^{j} < \infty
\]
when $p <1$. 
Since the sum is finite, we can select $k$ such that the initial type $x=(i,k+1)$ has an expected number of frozen $(0,0)$-type descendants strictly less than 1. By Markov's inequality there is a positive chance that the original particle has no $(0,0)$-type descendants, and hence has no  descendants in $\mathcal{L}_{i+1}$.
The family $\mathcal{A}\setminus\{\mathcal{L}_0\}$ satisfies the conditions of Theorem~\ref{th:finiteinfinite2}(v).



\emph{Assume $r<1$ (case {(i)}).}
For any $i\geq 1$, we show that if $r^i>p$, then $\bm{q}(\mathcal{L}_i) < \bm{q}(\mathcal{L}_{i+1})$, while if $r^i\leq p$ then $\bm{q}(\mathcal{L}_i) = \tilde{\bm{q}}$; this implies that $\mathcal{L}_i\Leftrightarrow\mathcal{L}_j$ if and only if $r^i\leq p$ and $r^j\leq p$. Hence the family $\mathcal{A}$ does not satisfy (C3), but the subfamily $\mathcal{A}':=\{\mathcal{L}_1, \ldots, \mathcal{L}_{i^*}\}$ does and it is regular.

Assume first that $r^i>p$. We need to show that there exists $x \in \mathcal{X}$ such that $q_x(\mathcal{L}_i) < q_x(\mathcal{L}_{i+1})$. Following similar arguments as in case \emph{(ii)}, it suffices to show that
the expected number of frozen $(0,0)$-type descendants of an initial $(i,1)$-type individual is finite. This expected number is bounded above by
\[
\sum_{j =1}^\infty \binom{i+j-1}{j-1} (r^{-j+1})^i\,p^j,
\]
which is finite when $r^i>p$.

Finally, assume $r^i\leq p$. Because $\bm{q}(\mathcal{L}_i)\leq \tilde{\bm{q}}$,  it remains to show that $\tilde{\bm{q}}\leq \bm{q}(\mathcal{L}_i)$, or equivalently, that survival in $\mathcal{L}_i$ implies local survival. Without loss of generality, we consider an initial $(i,1)$-type individual 
and we show that, with probability 1, it has an infinite number of $(0,0)$-type descendants.
%
Indeed, with probability 1, the initial individual has an infinite line of descendance made of type $\{(i,j)\}_{j \geq 1}$ individuals. The probability that any $(i,j)$-type individual in this line of descendants has at least one
(frozen) $(0,0)$-type descendant is bounded from below by the probability of having at least one descendant along the direct path from $(i,j)$ to $(0,j)$ in $\mathcal{P}_j$ and then along the direct path from $(0,j)$ to $(0,0)$ in $\mathcal{L}_0$.
This probability is $1-G_j^{(i)}(1-p^j)$, where $G_j^{(i)}(s)$ is the composition of $i$ geometric probability generating functions with mean $r^{-j+1}$ and satisfies
$$\dfrac{1}{1-G_j^{(i)}(s)}=\dfrac{(r^{j-1})^i}{1-s}+\left(1+r^{j-1}+(r^{j-1})^2\ldots+(r^{j-1})^{i-1}\right).$$
Because
\[
\sum_{j \geq 1} 1-G_j^{(i)}(1-p^j) = \sum_{j \geq 1} \left(r^{-1} \left(\frac{r^i}{p}\right)^j+\left(\frac{1-(r^i)^{j-1}}{1-r^{j-1}}\right)\right)^{-1}=\infty,
\]since the general term of the series diverges when $r^i\leq p$,
by the Borel-Cantelli lemma, with probability 1, the $(i,1)$-type individual has infinitely many (frozen) $(0,0)$-type descendants. By extension the same is true when we start with any $(i,j)$-type individual. This shows that survival in $\mathcal{L}_i$ implies local survival.
  In this case $\Sigma(\mathcal{A})=\Sigma(\{\mathcal{L}_i\colon 1\le i\le i^*\}$ and Theorem~\ref{th:finiteinfinite2}(i) applies.


\emph{Assume $r> 1$ (case {(iii)}).}
We show that for each $i,j \geq 1$,
\begin{equation}\label{eqcor}
\mbP_{(i,j)}( \mathcal{S}(\mathcal{L}_i) \cap \mathcal{E}(\mathcal{X} \backslash \mathcal{L}_i) )>0.\end{equation}
Corollary~\ref{th:2} then implies that there are uncountably many distinct extinction probability vectors. 


Recall that, with probability one, an initial $(i,j)$-type individual has an infinite line of descent made of types $(i,k)$ for $k>j$, and that $\xi_k$ denotes the geometric number of $(i-1,k)$-type offspring born to the $(i,k)$-type individual in this line of descent. 
%
By direct computation,
\[
 \pr(\xi_k=0\; \forall k \ge j)=
 \prod_{k=j}^{+\infty} \Big (1-\frac{r^{-k+1}}{1+r^{-k+1}} \Big )>0
\]
since $\sum_{k=j}^{+\infty} r^{-k+1}<+\infty$.
Thus, for all $j \ge 1$, there is positive probability that the descendants of $(i,j)$ never reach $\mathcal{L}_{i-1}$, and therefore \eqref{eqcor} holds.

\end{proof}

\begin{proof}[Proof of Proposition~\ref{prop:f}] A similar argument as in the proof of Proposition \ref{pro:ultimate} ($r>1$) can be  used  to show that $\mathcal L'_i \nLeftrightarrow \mathcal L'_j$ for all $i \neq j$.

Then, for any $I \subseteq \mbN$ with $|I|=\infty$ we have $\bm{q}(\bigcup_{i \in I} \mathcal L'_i) = \bm{q}(\mathcal{X})$, and for any $|I|<\infty$ we have $\bm{q}( \bigcup_{i \in I} \mathcal L'_i) = \bm{q}(\bigcup_{i\in {I}} \mathcal L_i )$; since the number of finite subsets of $\mbN_0$ is countably infinite, this proves that  {Ext}$(\mathcal{A}')$ is countably infinite.\end{proof}

\section*{Appendix A: Numerical computation of $\vc q(A)$}

We describe an iterative method to compute the extinction probability vector $\vc q(A)$ for any subset $A\subseteq \mathcal{X}$ in an irreducible MGWBP. Since $\mathcal{X}$ is countably infinite, we first relabel the types in $A$ as $1,2,3,4,\ldots$, and the types in $\mathcal{X}\setminus A$ as $1',2',3',4',\ldots$. For $k,\ell'\geq 1$, we then define $\vc q^{(k,\ell')}(A)$ as the global extinction probability vector of the  finite-type modified branching process where the types in $A$ larger than $k$ are \emph{immortal }and the types in $\mathcal{X}\setminus A$ larger than $\ell'$ are \emph{sterile}. More precisely, the offspring generating function $\bar{\vc G}^{(k,\ell')}(\vc s)$ of the modified process is such that
$$\begin{array}{lll}\bar{G}_i^{(k,\ell')}(\vc s)=G_i(\vc s) &\textrm{for all $i\in A$,} &i< k,\\\bar{G}_{i'}^{(k,\ell')}(\vc s)=G_{i'}(\vc s) &\textrm{for all $i'\in \mathcal{X}\setminus A$,} & i'< \ell',\\\bar{G}_i^{(k,\ell')}(\vc s)=0 &\textrm{for all $i\in A$,} & i\geq k,\\\bar{G}_{i'}^{(k,\ell')}(\vc s)=1 &\textrm{for all $i'\in \mathcal{X}\setminus A$,} & i'\geq \ell',\end{array}$$and $\vc q^{(k,\ell')}(A)$ is the minimal fixed-point  of the (finite) system $\vc s=\bar{\vc G}^{(k,\ell')}(\vc s)$, obtained by functional iteration.

\begin{proposition}If the MGWBP is irreducible then
$$\lim_{k\rightarrow\infty}\lim_{\ell'\rightarrow\infty}\vc q^{(k,\ell')}(A)=\vc q(A).$$
\end{proposition}
The proof follows the same arguments as that of Theorem 4.3 in \cite{Bra18}. Note that the convergence rate of the sequence $\{\vc q^{(k,\ell')}(A)\}_{k,\ell'\geq 1}$ depends on the way the types are relabelled. In addition, it is often more efficient to let $k=\ell'$ and let them increase to infinity together; however, we must be careful since that does not always guarantee convergence, as highlighted in  \cite{Bra18}. The computational method can be optimised depending on the example under consideration. 

%
%
%

\section*{Acknowledgements}

Daniela Bertacchi and Fabio Zucca acknowledge support from INDAM-GNAMPA and PRIN Grant 20155PAWZB. 
Peter Braunsteins has conducted part of the work while supported by the Australian Research Council (ARC) Laureate Fellowship FL130100039 and the Netherlands Organisation for Scientific Research (NWO) through Gravitation-grant NETWORKS-024.002.003. Sophie Hautphenne would like to thank the Australian Research Council (ARC) for support through her Discovery Early Career Researcher Award DE150101044 and her Discovery Project DP200101281.
The authors also acknowledge the ARC Centre of Excellence for Mathematical and Statistical Frontiers (ACEMS) for supporting the research visit of Daniela Bertacchi and Fabio Zucca at The University of Melbourne, during which this work was initiated.


\begin{thebibliography}{99}

\bibitem{ka2002}
D.E. Axelrod and  M. Kimmel,  
\newblock Branching Processes in Biology, 
\newblock Springer, New York, 2002.

\bibitem{cf:BZ}
D.~Bertacchi and F.~Zucca,
\newblock Critical behaviours and critical values of branching random walks
on multigraphs, 
\newblock \emph{J.~Appl.~Probab}.~\textbf{45}: 481--497, 2008.

\bibitem{cf:BZ12}
D.~Bertacchi and F.~Zucca, 
\newblock Recent results on branching random walks, 
\newblock \emph{Statistical Mechanics and Random Walks:
Principles, Processes and Applications}, Nova Science Publishers, 289--340, 2012.

\bibitem{cf:BZ14-SLS}
D.~Bertacchi and F.~Zucca,
\newblock Strong local survival of branching random walks is not monotone,
\newblock \emph{Adv.~Appl.~Probab.}~\textbf{46}(2): 400--421, 2014.

\bibitem{cf:BZ2015}
D.~Bertacchi and F.~Zucca,
\newblock A generating function approach to branching random walks,
\newblock \emph{Braz.~J.~Probab.~Stat.}~\textbf{31}(2): 229--253, 2017.

\bibitem{cf:BZ2020}
D.~Bertacchi and F.~Zucca,
\newblock Branching random walks with uncountably many extinction probability vectors,
\newblock \emph{Braz.~J.~Probab.~Stat.}~\textbf{34}(2): 426--438, 2020.

\bibitem{Big91} {J.D. Biggins, B.D. Lubachevsky, A. Shwartz, and A. Weiss,} 
\newblock A branching random walk with a barrier.
\newblock \emph{Ann. Appl. Probab.} 573--581, 1991.

\bibitem{Bra19} P.~Braunsteins and S.~Hautphenne,
\newblock Extinction in lower Hessenberg branching processes with countably many types.
\newblock \emph{Ann. Appl. Probab.}~\textbf{29}(5) : 2782--2818, 2019.

\bibitem{Bra18} P.~Braunsteins, S.~Hautphenne,
\newblock The probabilities of extinction in a branching random walk on a strip. 
\newblock \emph{J.~Appl.~Probab.}  \textbf{57}(3): 811--831, 2020.

\bibitem{Bra16} {P. Braunsteins, G. Decrouez, and S. Hautphenne.} \newblock A pathwise approach to the extinction of branching processes with countably many types. \newblock \emph{Stoch. Process. Their Appl.} \textbf{129}(3),713--739, 2019.

\bibitem{Can15}
E. Candellero and M. I. Roberts. 
\newblock
The number of ends of critical branching random walks.  \newblock \emph{ALEA:  Latin American Journal of Probability and Mathematical Statistics}, \textbf{12}(1), 55-67, 2015.

\bibitem{Com2007}
F. Comets and S. Popov. 
\newblock On multidimensional branching random walks in random environment. 
\newblock \emph{Ann. Probab}. \textbf{35}(1): 68--114, 2007.

\bibitem{Gan05} 
{N. Gantert and S. M{\"u}ller.} \newblock The critical branching Markov chain is transient. \newblock {\em Markov Proc.~and~rel.~Fields.}, {\bf 12}(4), 805--814, 2006.

\bibitem{cf:GMPV09}
N.~Gantert, S.~M\"uller, S.Yu.~Popov, and  M.~Vachkovskaia,
\newblock Survival of branching random walks in random environment,
\newblock {\em  J.~Theoret.~Probab.}~\textbf{23}(4) (2010), 1002--1014, 2010. 


\bibitem{harris}
T. E. Harris,
\newblock The theory of branching process,
1964. 

\bibitem{haut2013}
S. Hautphenne, G. Latouche and G. Nguyen. 
\newblock Extinction probabilities of branching
processes with countably infinitely many types.
\newblock \emph{Adv.~Appl.~Probab.} \textbf{45}(4):1068--1082, 2013.

\bibitem{Hut20}
T. Hutchcroft.
Transience and recurrence of sets for branching random walk via non-standard stochastic orders. \emph{arXiv preprint arXiv:2011.06402}, 2020.

\bibitem{Machado2003}
F.P. Machado and S.Y. Popov. 
\newblock Branching random walk in random environment on trees. \newblock \emph{Stoch. Process. Their Appl.}  \textbf{106}(1):95--106, 2003.

\bibitem{cf:MenshikovVolkov}
M.~V.~Menshikov and S.~E.~Volkov, 
\newblock Branching Markov chains: Qualitative characteristics,
\newblock {\em Markov Proc.~and~rel.~Fields.}~\textbf{3}, 225--241, 1997.

\bibitem{Moy62} J. E.~Moyal, 
\newblock Multiplicative population chains.
\newblock \emph{Proceedings of the Royal Society of London. Series A. Mathematical and Physical Sciences} \textbf{266}(1327): 518--526, 1962.

\bibitem{cf:Muller08-2}
S.~M\"uller, 
\newblock Recurrence for branching Markov chains,
\newblock \emph{Electron.~Commun.~Probab.}~\textbf{13}, 576--605, 2008.


\bibitem{shi}
Z.~Shi, Branching random walks,
volume 2151 of \emph{Lecture Notes in Mathematics.} Springer, Cham, 2015.

\bibitem{Spa89}  {A. Spataru},  \newblock Properties of branching processes with denumerably many types. \newblock {\em Revue Roumaine de Math\'ematiques Pures et Appliqu\'ees (Romanian Journal of Pure and Applied Mathematics)}, {\bf 34}, 747--759, 1989.

\bibitem{Stacey2003}
A. Stacey.
Branching random walks on quasi-transitive graphs.
\newblock \emph{Combin. Probab. Comput.} \textbf{12}(3):  345--358, 2003.

\bibitem{cf:Woess}
        W.~Woess, 
        \newblock {Random walks on infinite graphs and groups},
        Cambridge Tracts in Mathematics, {\textbf{138}},
    Cambridge Univ.~Press, 2000.
    
\bibitem{cf:Z1}
F.~Zucca,
\newblock  Survival, extinction and approximation of discrete-time branching random walks,
\newblock  {\em J.~Stat.~Phys.} \textbf{142}(4): 726--753, 2011.

\end{thebibliography}
\end{document}